\documentclass[11pt,
abstract=on
]{scrartcl}

\usepackage[utf8]{inputenc}
\usepackage[T1]{fontenc}
\usepackage[english]{babel}
\usepackage{amsmath,amssymb,amsfonts,siunitx,commath}
\usepackage{amsthm}
\usepackage{tikz,url}
\usepackage{geometry}
\usepackage{mathtools}
\mathtoolsset{showonlyrefs}
\usepackage{subcaption}
\usepackage{algorithm}
\usepackage{algpseudocode}
\usepackage{enumerate}
\usepackage{listings}
\usepackage{hyperref}
\usepackage[symbol]{footmisc}
\makeatletter

\DeclareMathOperator*{\argmax}{arg\,max}
\DeclareMathOperator*{\argmin}{arg\,min}

\newcommand{\R}{\mathbb{R}}
\newcommand{\C}{\mathbb{C}}

\newcommand{\N}{\mathbb{N}}

\newcommand\prox{\mathrm{prox}}

\newcommand\St{\mathrm{St}}

\newcommand\Toep{\mathrm{Toep}}
\newcommand{\tT}{\mathrm{T}}
\newcommand{\blkdiag}{\mathrm{diag}}
\DeclareMathOperator{\Circ}{Circ}
\DeclareMathOperator*{\diag}{diag}
\newcommand{\sym}{\textnormal{Sym}}

\theoremstyle{plain}
\newtheorem{lemma}{Lemma}
\newtheorem{thm}[lemma]{Theorem}

\newtheorem{proposition}[lemma]{Proposition}
\theoremstyle{definition}

\newtheorem{example}[lemma]{Example}

\newtheorem{remark}[lemma]{Remark}

\newcommand{\toprule}{\hrule height.8pt depth0pt \kern2pt} 
\newcommand{\midrule}{\kern2pt\hrule\kern2pt} 
\newcommand{\bottomrule}{\kern2pt\hrule\relax}
\newcommand{\algcaption}[2][]{%
  \refstepcounter{algorithm}%
    {\addcontentsline{loa}{figure}{\protect\numberline{\thealgorithm}{\ignorespaces #2}}}
    {\addcontentsline{loa}{figure}{\protect\numberline{\thealgorithm}{\ignorespaces #1}}}%
  \toprule
  \textbf{Algorithm~\thealgorithm}\ #2\par 
  \midrule
}
\makeatother

\begin{document}
\title{Convolutional Proximal Neural Networks
\\
and Plug-and-Play Algorithms}

\author{
Johannes Hertrich\footnotemark[1]
\and
Sebastian Neumayer\footnotemark[1]
\and
Gabriele Steidl\footnotemark[1]
	}

\maketitle             
\footnotetext[1]{
Institute of Mathematics,
	TU Berlin,
	Stra{\ss}e des 17.~Juni 136, 
	10623 Berlin, Germany,
	\{j.hertrich,neumayer,steidl\}@math.tu-berlin.de.
	}

\begin{abstract}
In this paper, we introduce convolutional proximal neural networks (cPNNs), which are by construction averaged operators.
For filters of full length, we propose a stochastic gradient descent algorithm on a submanifold of the Stiefel manifold to train cPNNs.
In case of filters with limited length, we design algorithms for minimizing functionals that  approximate the orthogonality constraints imposed on the operators by penalizing the least squares distance to the identity operator.
Then, we investigate how scaled cPNNs with a prescribed Lipschitz constant can be used for denoising signals and images, where the achieved quality depends on the Lipschitz constant.
Finally, we apply cPNN based denoisers within a Plug-and-Play (PnP) framework and provide convergence results for the corresponding PnP forward-backward splitting algorithm based on an oracle construction.
\end{abstract}

\section{Introduction} \label{sec:intro}
Building neural networks (NNs) with special properties and stability guarantees has attracted growing interest over the last few years.
In particular, it turned out that controlling the Lipschitz constant of neural networks \cite{GFPC18,MKKY2018,SGL2019} is an important step towards increasing robustness, e.g., against adversarial attacks \cite{TSS2018}.
In this paper, we are interested in NNs that are averaged with parameter $t \in (0,1)$, see \cite{Kr55,Ma53} for the first attempt at averaged operators.
Of special interest are operators with parameter $t = \frac12$. 
These include so-called proximal operators, which were extensively used 
in variational image processing lately \cite{BPCPE11}.
Indeed, averagedness of an operator requires more than just having Lipschitz constant 1, and the convergence of certain iteration schemes relies on this property.
A well-known example is the iterative soft thresholding algorithm (ISTA) \cite{DDM04}, which is itself a special case of forward-backward splitting (FBS) algorithms \cite{CW05}. 
Convergence of the Douglas--Rachford  algorithm 
and the alternating direction method of multipliers (ADMM) can be shown using properties of averaged operators.
The relation between the later algorithms is investigated in \cite{CKCH2020,EB92,Se2011}.

Recently, a certain proximal step in these algorithms was replaced by powerful denoisers such as BM3D \cite{DFKE2007,DKE2012} or NNs \cite{SWK2019} without any convergence guarantees.
This technique is meanwhile known as Plug-and-Play (PnP) algorithms \cite{SVW2016,VBW13} and led to improved results in certain applications, e.g., \cite{CWE2016,GJNMU2018,MMHC2017,Ono2017,TBF2017}.
The idea was also applied for various other optimization methods such as half-quadratic minimization \cite{ZZGZ2017} or the primal dual hybrid gradient algorithm \cite{MMHC2017}.
Using slightly different ideas, PnP iterations based on the D-AMP algorithm are constructed in \cite{DMM2009}.
Here, a separate denoiser is trained for each layer, an approach also known as algorithmic unrolling \cite{MLE19}. 	
To remain close to proximal operators as denoiser choice for FBS-PnP, the authors of \cite{CLPKS2017} proposed a NN trained to be the orthogonal projection onto the set of natural images. 	
Closely related to PnP is an approach called regularization by denoising (RED) \cite{REM2017}.
In this setting, an objective function of the form
$\|Ax-y\|^2+\rho(x)$ is minimized, where the regularizer $\rho$ is built from a NN denoiser.
Unfortunately, convergence of FBS-PnP can be only guaranteed if the denoiser is averaged, and, even worse, convergence of ADMM-PnP if the denoiser is $\frac12$-averaged.
These requirements are clearly not fulfilled by BM3D and general NNs.
Indeed, it was shown in \cite{SKM2019} that we do not have numerical convergence within FBS-PnP for DnCCN \cite{{ZZCMZ2017}}. 
Additionally, we construct for any $t\in(\tfrac12,1]$ a $t$-averaged operator such that ADMM-PnP diverges.
In \cite{HHNPSS2019}, so-called PNNs were considered, which are by construction averaged and hence a natural choice for PnP algorithms.
Note that averaged NNs can be also built using the definition of averagedness directly, see \cite{TRPW20}.

In this paper, we generalize the framework of PNNs from \cite{HHNPSS2019} to convolutional ones.
We want to emphasize that this generalization is vital for working with real-world image data.
To this end, we first cover the theoretical background on the structure of the minimization domain, 
before we discuss the actual training procedure.
It turns out that the case of full filter length can be handled by considering
submanifolds of the Stiefel manifold.
For filters with limited length, the situation is completely different and we propose a penalized version of the problem instead.
In order to improve the expressibility of the constructed NNs, we scale the cPNNs with a factor $\gamma$, which is an upper bound for their Lipschitz constant.
Our numerical experiments confirm that this scaling indeed leads to better denoisers.
To ensure that the networks remain averaged, we have to modify them using an additional oracle image.
Based on this modification, we can show that the proposed FBS-PnP schemes are convergent.
For smaller noise levels, we observe that our scaled cPNNs are numerically $\frac12$-averaged,
which justifies their usage in ADMM-PnP algorithms.

The outline of this paper is as follows:
In Section~\ref{sec:prelim}, 
we recall tools from convex analysis and facts about Stiefel manifolds that are necessary to derive the stochastic gradient descent algorithm on these manifolds.
The construction of PNNs as proposed in \cite{HHNPSS2019} is briefly reviewed in Section~\ref{sec:PNN}.
Next, this framework is extended to the convolutional setting, called cPNNs, in Section~\ref{sec:cPNN}.
Here, we first investigate filters with full length, 
before we resort to filters with limited length in the second part.
In Section~\ref{sec:DenoisePnNN}, we show how scaled cPNNs with an a priori upper-bounded Lipschitz constant can be used to construct image denoisers. 
As we do not want to retrain a network for every noise level, we propose to use cPNN based denoisers within PnP algorithms in Section~\ref{sec:PnP}.
Here, it is important that the denoisers have not only Lipschitz constant 1, but are actually averaged.
For an oracle version of the denoiser, we prove convergence of the PnP-FBS algorithm.
Numerical results of PnP methods with cPNNs for image denoising and deblurring are provided in Section~\ref{sec:numerics}.
Finally, conclusions are drawn in Section~\ref{sec:conclusions}.

\section{Preliminaries} \label{sec:prelim}
Throughout this paper, we denote by $I_d$ the $d \times d$ identity matrix, by
$1_d$ the $d$-dimensional vector containing only 1s, and by $\|\cdot\|$ the Euclidean norm.

\paragraph{Convex analysis}
In the following, $\Gamma_0(\R^d)$ denotes the set of proper, convex, lower semi-continuous 
functions on $\R^d$ mapping into $(-\infty,  \infty]$. 
For $f \in \Gamma_0(\R^d)$, the \emph{proximity operator} $\prox_{f}\colon \R^d \rightarrow \R^d$ 
is defined by
\begin{align}\label{prox_usual}
\prox_{f} (x) 
&\coloneqq 
\argmin_{y \in \R^d} \bigl\{ \tfrac12 \|x-y\|^2 +  f(y) \bigr\}.
\end{align}
It was shown by Moreau \cite[Cor.~10c]{Moreau65} that
$A \colon  \R^d \rightarrow \R^d$ is a proximity operator of some function $f \in \Gamma_0(\R^d)$
if and only if it is non-expansive 
and has a potential $\varphi \in \Gamma_0(\R^d)$, i.e.,
 $A(x) = \nabla \varphi(x)$ for all $x \in \R^d$.

In this work, we focus on NNs that are so-called averaged operators.
Recall that an operator $A\colon \R^d \rightarrow \R^d$ is \emph{$t$-averaged} if there exists a non-expansive 
operator $R\colon\R^d \rightarrow \R^d$ such that
$$A = t R + (1-t) I_d \quad \mathrm{for\;  some} \quad t \in (0,1).$$
Clearly, averaged operators are non-expansive.
The averaged operators $A \colon  \R^d \rightarrow \R^d$ with $t=\frac12$ are exactly the \emph{firmly non-expansive operators}
fulfilling
$$
\|Ax - Ay\|^2 \le \langle Ax - Ay,x-y \rangle
$$
for all $x,y \in \R^d$.
In particular, every proximity operator is firmly non-expansive. 
Using Moreau's characterization, it can be seen that the opposite is in general not true.
However, for functions on $\mathbb R$ we have the following proposition, which follows immediately from Moreau's result
\cite{CP2018}.

\begin{proposition}\label{prop:stab}
For a  function $\sigma\colon \mathbb R \rightarrow \mathbb R$ the following
properties are equivalent:
\begin{itemize}
\item[i)] $\sigma = \prox_{f}$ for some $f \in \Gamma_0(\mathbb R)$,
\item[ii)] $\sigma$ is $\frac12$-averaged,
\item[iii)] $\sigma$ is monotone increasing and non-expansive. 
\end{itemize}
\end{proposition}

Some useful  properties of averaged operators are given in the following theorem.
The first three parts follow directly by definition and the last two can be found in \cite{CY15}.

\begin{thm}[Properties of averaged operators]\label{alpha_lin}
\hspace{1em}
\begin{itemize}
\item[i)] If $A\colon \R^{d} \rightarrow \R^{d} $ is Lipschitz continuous with Lipschitz constant $L < 1$, then
$A$ is averaged for every parameter $t \in [(L+1)/2,1)$.
\item[ii)] If $A\colon \R^{d} \rightarrow \R^{d} $ is averaged with parameter $t$, then it is averaged for
every parameter in $[t,1)$.
\item[iii)] A linear operator $A\in \R^{d,d}$ is averaged with parameter $t \in (0,1)$
if and only if the matrix
$
(2t-1) I_d - A^\tT A + (1-t) (A + A^\tT)
$
is positive semidefinite.
In particular, any symmetric $A$ is averaged if and only if all its eigenvalues are in $(-1,1]$,
and it is $\frac12$-averaged if all its eigenvalues are in $[0,1]$.
\item[iv)] The concatenation of $K$ averaged operators $A_k\colon \R^{d} \rightarrow \R^{d}$ with parameters $t_k \in (0,1)$, $k=1,\ldots,K$,
is an averaged operator with parameter 
$$
t = 
\Big(1+ \big(\sum_{k=1}^K \frac{t_k}{1-t_k}\big)^{-1} \Big)^{-1} 
\le
\frac{K}{(K-1) + \frac{1}{\max_k t_k} }.
$$
\item[v)]
If $A\colon \R^{d} \rightarrow \R^{d}$ 
is an averaged operator having a nonempty  fixed point set,
then the sequences of iterates $\{x^{(r)}   \}_{r}$ generated by $x^{(r+1)} = A x^{(r)}$
converges for every starting point $x^{(0)}$ to a fixed point of $A$.
\end{itemize}
\end{thm}

\paragraph{Stiefel manifold}
For $n \ge d$, the (compact) \emph{Stiefel manifold} is defined as 
$$
\St(d,n) \coloneqq \bigl\{ T \in \mathbb R ^{n,d} : T^\tT T = I_d\bigr\}.
$$
This manifold has dimension $nd -\frac12 d(d+1)$ and its
tangent space at $T \in \mathrm{St}(d,n)$ is given by 
$$
{\mathcal T}_T \mathrm{St}(d,n) =\bigl\{T V + B: V^\tT = -V, T^\tT B = 0\bigr\},
$$
see \cite{AMS08}.
For fixed $T \in \St(d,n)$, the \emph{orthogonal projection} of $X \in\R^{n,d}$ 
\emph{onto the tangent space} ${\mathcal T}_T \mathrm{St}(d,n)$ is given by
\begin{align} \label{proj_stiefel_1}
\Pi_{{\mathcal T}_T \mathrm{St}(d,n)} X 
&= (I_n - T T^\tT) X + \tfrac12 T(T^\tT X - X^\tT T)\\
&= W(T,X) \, T,
\end{align}
where
\begin{equation} \label{proj_stiefel_2}
W(T,X) \coloneqq \hat W (T,X) - \hat W^\tT (T,X), \quad \hat W(T,X)  \coloneqq X T^\tT-\tfrac12 T(T^\tT X T^\tT).
\end{equation}
Several retractions on $\St(d,n)$ based on the QR factorization of matrices  were proposed in the literature.
As the computation of the QR decomposition appears to be time consuming on a GPU, 
we prefer to apply the following retraction \cite{NA2005,WY2013} based on the Cayley transform of skew-symmetric matrices
\begin{equation} \label{retraction_2}
{\mathcal R}_{\mathcal T_T \St(d,n)}(X) = \bigl(I_n-\tfrac12 W(T,X) \bigr)^{-1}\bigl(I_n +\tfrac12 W(T,X)\bigr)T, \quad X \in {\mathcal T}_T\mathrm{St}(d,n).
\end{equation}
Note that the required matrix inversion can be efficiently evaluated by a fixed point iteration, see \cite{HHNPSS2019,LLT2020}.
Since
$$W(T,X) = W(T,\Pi_{\mathcal T_{T}\St(d,n)} X),$$ 
this retraction enlarged to the whole $\mathbb R^{n,d}$ fulfills
\begin{equation}\label{retraction-projection}
{\mathcal R}_{\mathcal T_T \St(d,n)} (X) = {\mathcal R}_{\mathcal T_T \St(d,n)} (\Pi_{\mathcal T_T \St(d,n)} X), 
\quad X \in \mathbb R^{n,d}.
\end{equation}

Finally, we need the orthogonal projection with respect to the Frobenius norm
onto the Stiefel manifold itself, see \cite[Sect.~7.3, 7.4]{HS2013}.
An orthogonal projection of $X \in \mathbb R^{n,d}$, $d \le n$, onto $\St(d,n)$ is 
given by the matrix $U$ in the polar decomposition 
\begin{equation} \label{polar}
X=U S, \quad U\in \St(d,n),\, S\in \R^{d,d} \; \mathrm{symmetric \; positive \; definite }.
\end{equation}
The projection is unique if and only if $X$ has full column rank.

\section{Proximal Neural Networks} \label{sec:PNN}
In this section, we consider proximal neural networks introduced by some of the authors in \cite{HHNPSS2019}.
Note that the original network structure is slightly more general than the one used in this paper.
Such networks are concatenations of building blocks of the form
\begin{equation}\label{building_block}
\Phi(\cdot\,; T,b,\alpha)  \coloneqq T^\tT \sigma_\alpha (T \cdot + b) , 
\end{equation} 
where $b \in \R^n$, the matrix $T$ or $T^\tT$ is in $\St(d,n)$ and $\sigma_\alpha$ is a \emph{stable activation function}, i.e., a function which satisfies $\sigma_\alpha(0) = 0$ and one of the
equivalent properties in Proposition~\ref{prop:stab}.
The activation function may depend on a parameter $\alpha >0$ or is 
parameter-free in which case we ignore the index $\alpha$.
Nearly all practically applied activation functions are stable, see Table \ref{prox:act} in the appendix.
For $T \in \R^{n,d}$, we use the abbreviation
$$
T' \coloneqq 
\left\{
\begin{array}{ll}
T &\mathrm{if} \; n \ge d,\\
T^\tT &\mathrm{if} \; n < d.
		\end{array}
		\right.
$$
Under the above assumptions, it can be shown using Moreau's characterization that the function $\Phi$ in \eqref{building_block} is indeed a proximity operator 
of some function from $\Gamma_0(\mathbb R^d)$, see \cite{Beck17,HHNPSS2019}.
In particular, $\Phi$ is $\frac12$-averaged.
A \emph{proximal neural network} (PNN) is defined as
\begin{align}\label{eq_PPNN}
\mathbf{\Phi} (x;u) = T_{K}^\tT \sigma_{\alpha_{K}}(T_{K}\cdots T_2^\tT \sigma_{\alpha_2}(T_2T_1^\tT \sigma_{\alpha_1}(T_1x+b_1)+b_2)\cdots),
\end{align}
with parameters 
$u = ( (T'_k)_{k=1}^{K},(b_k)_{k=1}^{K},(\alpha_k)_{k=1}^{K})$,
where 
$T'_k \in \St(d,n_{k})$, $b_k \in \R^{n_k}$ and $\alpha_{k}>0$, see \cite{HHNPSS2019}.
PNNs are the concatenation of $K$ firmly non-expansive operators. Hence, Theorem~\ref{alpha_lin}iii) implies that they are averaged with $\alpha = K/(K+1)$. 

\begin{remark}
The authors of
\cite{huang2018orthogonal} considered NNs of the form
$$
L_{K+1} \sigma_{\alpha_{K}}(L_{K}\cdots L_3 \sigma_{\alpha_2}(L_2 \sigma_{\alpha_1}(L_1x+b_1)+b_2)\cdots),
$$
where the matrices $L_k \in R^{n_k,n_{k-1}}$ or their transposed are in a Stiefel manifold and trained them by the so-called optimization over multiple dependent Stiefel manifolds (OMDSM).
It can be shown that for a given $T'_{k-1} \in \St(d,n_{k-1})$ there exists
$T'_{k} \in \St(d,n_{k})$ 
such that
$L_k = T_k T_{k-1}^\tT$, see \cite{HHNPSS2019}.
The converse is not true, i.e., $T_k T_{k-1}^\tT$ (or its transpose) is in general not in  a Stiefel manifold.
Thus,  PNNs  are more general than the NNs in \cite{huang2018orthogonal}.
\end{remark}

In the following, we want to learn the parameters $u=(T,b,\alpha)$ of PNNs, where
\begin{itemize}
\item[-]
$T=(T'_k)_{k=1}^{K} \in \mathcal{U}_2 \coloneqq \St(d,n_1) \times \ldots \times \St(d,n_{K})$,
\item[-]
$b = (b_k)_{k=1}^{K} \in \mathcal{U}_1 \coloneqq \R^{n_1} \times \ldots \times \R^{n_{K}}$,
\item[-]
$\alpha = (\alpha_k)_{k=1}^{K} \in \mathcal{U}_0 \coloneqq \R_{>0}^{K}$.
\end{itemize}
For learning the parameters, we have to solve for given $(x_i,y_i)_{i=1}^N$
the following minimization problem
\begin{align}\label{eq:opt_problem}
\argmin_{(T,b,\alpha) \in \mathcal{U}_2 \times \mathcal{U}_1 \times \mathcal{U}_0}
\frac1N\sum_{i=1}^N\ell\bigl(\mathbf{\Phi}(x_i;u),y_i\bigr)
&=
\argmin_{u} \left\{ H(u) + \iota_{\mathcal{U}_2} (T) + \iota_{\mathcal{U}_0} (\alpha) \right\},
\end{align} 
with a differentiable  loss function $\ell$ and
$H(u) \coloneqq \frac1N\sum_{i=1}^N \ell\bigl(\Phi(x_i;u),y_i\bigr)$.
Here, $\iota_C$ denotes the \emph{indicator function} of a set $C$ defined by 
$\iota_C(x) \coloneqq 0$ if $x \in C$ 
and $\iota_C(x) \coloneqq +\infty$ otherwise.

In \cite{HHNPSS2019}, this functional was minimized  by a stochastic gradient decent algorithm on the manifold $\mathcal{U}_2 \times \mathcal{U}_1 \times \mathcal{U}_0$.
One  gradient descent step for minimizing a function $f\colon \mathcal M \rightarrow \R$ 
on a matrix manifold $\mathcal M$ embedded in $\R^D$ can be basically performed as
follows:
Given the previous iteration $\bar x \in \mathcal M$, 
\begin{itemize}
\item[i)]  compute the Euclidean gradient $\nabla f(\bar x)$ of $f$ at $\bar x$,
\item[ii)]  perform the orthogonal projection $\Pi_{\mathcal T_{\bar x}\mathcal M}\colon \R^D \rightarrow \mathcal T_{\bar x}\mathcal M$ 
of $\nabla f(\bar x)$ onto the tangent space $\mathcal T_{\bar x}\mathcal M$ of $\mathcal M$ at $\bar x$,
$$\Pi_{\mathcal T_{\bar x}\mathcal M} \left(\nabla f(\bar x) \right) = \nabla_{\mathcal M}  f(\bar x)$$ 
to obtain the Riemannian gradient 
$\nabla_{\mathcal M} f(\bar x)$,
\item[iii)] perform a descent step in direction $\nabla_{\mathcal M}  f(\bar x)$ using a retraction
${\mathcal R}_{\mathcal T_{\bar x} \mathcal M}\colon \mathcal T_{\bar x}\mathcal M \rightarrow \mathcal M$ to obtain a new iterate
$$
x_{\mathrm{new}} = {\mathcal R}_{\mathcal T_{\bar x} \mathcal M} \left( - \tau \nabla_{\mathcal M}  f(\bar x) \right), \quad \tau > 0.
$$
\end{itemize}
Since  $\mathcal{U}_1$  is  an Euclidean space, we only have to consider orthogonal projections and retractions on $\St(d,n)$ and on the positive numbers.
On $\St(d,n)$ we use the retraction \eqref{retraction_2}, which has the additional advantage that step ii) in Algorithm~\ref{alg:SGD_Stiefel} is not necessary, i.e., we can apply the enlarged retraction \eqref{retraction-projection} to the Euclidean gradient directly.

For learning the positive parameters of the activation function,
we have to deal with the Riemannian manifold $\R_{>0}$ 
with tangent space $\mathcal T_\alpha\R_{>0}=\R$ at $\alpha >0$
and Riemannian metric 
$\langle r,s\rangle_\alpha=\tfrac{rs}{\alpha^2}$ with associated distance 
$\text{dist}(\alpha,\beta) = \vert \ln (\alpha / \beta) \vert$.
As retraction, we use the exponential map  given by 
$\mathrm{EXP}_\alpha(r) = \alpha \exp(r/\alpha)$.
Now, the Riemannian gradient $\nabla_{\R_{>0}} f(\alpha)$ on the manifold $\R_{>0}$ fulfills for all $r\in\R$ that
$$
\tfrac1{\alpha^2}\nabla_{\R_{>0}}f(\alpha) r
=
\langle \nabla_{\R_{>0}}f(\alpha),r\rangle_\alpha = r f'(\alpha),
$$
where $f'$ is the Euclidean derivative of $f$. 
Thus, $\nabla_{\R_{>0}}f(\alpha) = \alpha^2 f'(\alpha)$ and a gradient descent step reads as
$$
\alpha^{(r+1)}
= \mathrm{EXP}_{\alpha^{(r)}} \bigl(-\nabla_{\R_{>0}} f(\alpha^{(r)}) \bigr) 
= \alpha^{(r)} \exp \bigl(-\alpha^{(r)} \, f'(\alpha^{(r)} )\bigr).
$$

In summary, the stochastic gradient descent reprojection algorithm on $\mathcal{U}_2 \times \mathcal{U}_1 \times \mathcal{U}_0$
is given in  Algorithm~\ref{alg:SGD_Stiefel}.
So far, we have learned fully populated matrices $T'_k\in \mathrm{S}(d,n_{k})$, $k=1,\ldots,K$.
To cope with real-world applications, we have to generalize our concept to convolutional networks.

\begin{algorithm}[!t]
\begin{algorithmic}
\State \textbf{Input:} Training data $(x_i,y_i)_{i=1}^N$, batch size $B\in\N$, 
learning rate $(\tau^{(r)})_{r\in\N}$
\State \textbf{Initialization:}  $\mathcal M \coloneqq \St(n_K,d) \times \ldots \times \St(n_1,d)$, $(T^{(0)},b^{(0)}, \alpha^{(0)})$, 
where $T^{(0)}$ has components $(T')_k^{(0)}$, $k=1,\ldots,K$.

\For{$r=0,1,\ldots$}
    \State 1. Choose a mini batch $I\subset\{1,\ldots,N\}$ of size $|I|=B$
    \State 2. Compute the Euclidean gradients of 
		$ H_I(u) \coloneqq \sum_{i\in I} \ell(\mathbf{\Phi}(x_i;u);y_i)$ 
		\State using backpropagation
    $$
		\nabla_T      H_I\bigl(u^{(r)}\bigr), \quad
    \nabla_b      H_I\bigl(u^{(r)}\bigr), \quad
    \nabla_\alpha H_I\bigl(u^{(r)}\bigr)
    $$
    \State 3. Update $u$ by a gradient descent step using the retraction \eqref{retraction-projection}
    \begin{align}
    T^{(r+1)} &=  {\mathcal R}_{\mathcal T_{T^{(r)}} \mathcal M } 
		\left( -\tfrac{\tau^{(r)}}{B}\nabla_{T} 
		H_I 
		\bigl(T^{(r)},b^{(r)}
		\bigr) \right),
		\\
    b^{(r+1)} &= b^{(r)}-\tfrac{\tau^{(r)}}{B}\nabla_b  H_I
		\bigl(T^{(r)},b^{(r)}\bigr)	,\\
		\alpha^{(r+1)} &= \alpha^{(r)}\exp\left(-\alpha^{(r)} \nabla_{\alpha} H_I\bigl(\alpha^{(r)}\bigr)\right)
    \end{align}		
\EndFor
\end{algorithmic}
\caption{Stochastic gradient descent algorithm for minimizing \eqref{eq:opt_problem} 
} \label{alg:SGD_Stiefel}
\end{algorithm}

\section{Convolutional Proximal Neural Networks} \label{sec:cPNN}
In many real world applications, in particular in image processing,
it is not possible to learn full matrices $T_k$, $k=1,\ldots,K$.
Therefore, we address the construction of PNNs with convolutional layers, called cPNNs, in this section.
In Subsection \ref{sec:cPNN+matrix}, we investigate convolutions having full filter lengths.
It turns out that in this case we are just dealing with a submanifold of the Stiefel manifold
so that similarly as before a stochastic gradient descent algorithm on this submanifold can be used for training.
However, for high-dimensional data it is preferable to learn sparse filters.
Therefore, we resort to cPNNs with filters of prescribed lengths in Subsection~\ref{sec:cPNN+limit}.
Unfortunately, we leave the submanifold setting here
and have to apply a completely different approach for training.
 
\subsection{Convolutional PNNs with Full Filter Lengths and Matrix Algebras}\label{sec:cPNN+matrix}
There exist several approaches to deal with boundaries when applying a filter $a \in \mathbb R^m$ to some finite signal
$f \in \mathbb R^n$, i.e., when computing the vector
$f*a = ( (f*a)_j)_{j=0}^{m-1}$
given by
$$
(f*a)_j = \sum_{k=0}^{m-1} a_k f_{j-k}, \quad j=0,\ldots,n-1,
$$
with a suitable extension of $f$ for indices not between $0$ and $m-1$.
Common extension choices are  
\begin{itemize}
\item[-] periodic: $f_{j+lm} = f_j$, 
\item[-] mirrored: $f_{2m-j-1} = f_j$ and $f_{j+2lm} = f_j$,
\item[-] zero-padding: $f_{j+lm} = 0$,
\end{itemize}
where $j\in \{0,\ldots,m-1\}$ and $l \in \mathbb Z \setminus \{0\}$.
The first two cases are related to the matrix algebras of circulant 
and special Toeplitz-plus-Hankel matrices, 
while the last one relies on Toeplitz matrices.
More generally, for a unitary matrix $U_m\in\C^{m,m}$, 
we consider the commutative matrix algebra
\[                                                        
\mathcal{A}(U_m) =\bigl\{ C=U_m \Lambda U_m^* \in \R^{m,m}:~\Lambda \; \text{complex diagonal matrix}\bigr\} .
\]
The mentioned boundary cases are treated in the next example, for more information see, e.g., \cite{PS99,strang2014functions}.

\begin{example}
	\hspace{1em}
	\begin{enumerate}
		\item \textbf{Filtering of periodic signals:}\\
		For the $m$-th Fourier matrix 
		$U_m = F_{m} = \frac{1}{\sqrt{m}}(e^{-2\pi \mathrm{i} jk/m})_{j,k = 0}^{m-1,m-1}$, 
		we obtain the algebra of circulant matrices  with first column $a \in \R^m$,                                    
		\begin{equation} \label{circs}
		\mathcal{A}(F_m) = \bigl\{\Circ_m(a) \coloneqq  F_m^* \diag( \sqrt{m}  F_m a) F_m: a \in \R^{m} \bigr\}.
		\end{equation}
		Filtering with periodic boundary conditions can be written as
		$f*a = \Circ_m(a) f$. 
		
		\item \textbf{Filtering of signals with mirrored boundaries:}\\
		For the cosine-II matrix
		\[U_m = C_m^{\mathrm{II}}=\sqrt{\frac{2}{m}}\Bigl(\varepsilon_k\cos\frac{(2k+1)j\pi}{2m}\Bigr)_{j,k=0}^{m-1}\]
		with   $\varepsilon_0=1/\sqrt{2}$ and $\varepsilon_k=1$ for $k=1,\ldots,m-1$,
		we get the algebra of symmetric Toeplitz matrices with first column $a$ plus persymmetric Hankel matrices
		with first column $(a_1,\ldots,a_{m-1},0)^\tT$,
		$$
		\mathcal{A}(C_m^{\mathrm{II}}) = \bigl\{\mathrm{TH}_m(a) \coloneqq (C_m^{\mathrm{II}})^\tT \diag(C_m a) C_m^{\mathrm{II}}: a \in \R^{m}\bigr\},
		$$ 
		where $C_m \coloneqq 2(\varepsilon_k^2 \cos(jk\pi/m))_{j,k=0}^{m-1}$.
		Filtering with mirror boundary condition is given by $f*a = \mathrm{TH}_m(a) f$.
		
		\item Another common algebra in the context of convolutions is the so-called $\tau$-algebra based on the sin-I transform, 
		see \cite{bini1983spectral}.
	\end{enumerate}
\end{example}

The matrices appearing in NNs usually consist of several blocks of filter matrices.
Therefore, we consider the linear subspace of $\R^{n,d}$ given by block matrices of structured blocks
$$
\mathcal{V}(U_m) ^{m_1,m_2} = \bigl\{ C=(C_{ij})_{i,j=1}^{m_1,m_2}: 
C_{ij} \in \mathcal{A}(U_m) \bigr\},
$$
with 
$$n = m_1 m, \quad d = m_2 m.$$
Special block circulant matrices are addressed in  Proposition \ref{thm:toep_orth_yields_circ_orth}.
Now, we investigate the subset
	\begin{align}\label{setA}
	\mathcal{M} (U_m) ^{m_1,m_2}
	&\coloneqq
	 \mathcal{V}(U_m) ^{m_1,m_2} \cap \mathrm{St}(d,n).
	\end{align}
More precisely, we show that $\mathcal{M} (U_m) ^{m_1,m_2}$ is a submanifold of $\mathrm{St}(d,n)$
and that the orthogonal projections onto its tangent spaces as well as the retractions
coincide with those of the Stiefel manifold when restricted to $\mathcal{V}(U_m) ^{m_1,m_2}$.
Here, the algebra property of $\mathcal{A}(U_m)$ plays a crucial role.
As a consequence, the stochastic gradient descent Algorithm~\ref{alg:SGD_Stiefel} can be applied on $\mathcal{M} (U_m) ^{m_1,m_2} \times \mathcal U_1 \times \mathcal U_0$ as well.
We need the following well-known lemma, see \cite{AMS08}.
	
	\begin{lemma}\label{lem:01}
		Let $\mathcal V_1, \mathcal V_2$ be linear spaces of dimensions $d_1 > d_2$.
		For a smooth function $F \colon \mathcal V_1 \to \mathcal V_2$ with full rank Jacobian $DF(x)$ for all $x\in \mathcal V_1$, the set
		\[F^{-1} (0) = \bigl\{ x \in \mathcal V_1 : F(x) = 0\bigr\},\]
		is a smooth submanifold of $\mathcal V_1$ of dimension $d_1-d_2$ and for any $x \in F^{-1} (0)$ it holds
		\[\mathcal T_x F^{-1} (0) = \ker\bigl(DF(x)\bigr), \quad x \in F^{-1} (0).\]
	\end{lemma}
	
Based on this lemma, we can prove the following proposition.
		
	\begin{proposition}\label{prop:submanifold}
			The set $\mathcal M = \mathcal{M} (U_m) ^{m_1,m_2}$ defined in \eqref{setA} is a manifold and the tangential space at $C \in \mathcal M$ 
		is given by
		\begin{align*}
		{\mathcal T}_C \mathcal M 
		&=
		\bigl\{C V + B: V \in \mathcal{V}(U_m)^{m_2,m_2}, 
		B \in\mathcal{V}(U_m)^{m_1,m_2} \text{ with } V^\tT = -V, \, C^\tT B = 0\bigr\}\\
		&=
		\mathcal{V}(U_m) ^{m_1,m_2} \cap T_C \mathrm{St}(d,n).
		\end{align*}
	The orthogonal projection of $X \in \mathcal{V}(U_m)^{m_1,m_2}$ onto the tangent space 
		${\mathcal T}_C \mathcal M $ is given by 
		$$ \Pi_{{\mathcal T}_C \mathcal M} (X) = P_{{\mathcal T}_C \mathrm{St}(d,n)} (X),$$
		and a retraction by	
		$$
		{\mathcal R}_{{\mathcal T}_C \mathcal M} (X) = {\mathcal R}_{{\mathcal T}_C \mathrm{St}(d,n)} (X),
		$$
		where ${\mathcal R}_{{\mathcal T}_C \mathrm{St}(d,n)}$ denotes the retraction in \eqref{retraction_2}.
		Further, we have	for the retraction enlarged to 	$\mathcal{V}(U_m)^{m_1,m_2}$ that
	\begin{equation}
	{\mathcal R}_{{\mathcal T}_C \mathcal M} (X) 
	= {\mathcal R}_{{\mathcal T}_C \mathrm{St}(d,n)} (\Pi_{{\mathcal T}_C \mathcal M} X), \qquad X \in \mathcal{V}(U_m)^{m_1,m_2}.
	\end{equation}
	 	\end{proposition}
	
	\begin{proof}
	1. In order to apply Lemma \ref{lem:01}, we consider $\mathcal V_1 = \mathcal{V}(U_m)^{m_1,m_2}$ and the space of symmetric matrices $\mathcal V_2 = \sym(\mathcal{V}(U_m)^{m_2,m_2}) \subset \mathcal{V}(U_m)^{m_2,m_2}$.
	As mapping we choose $F(C) = C^\tT C - I_{d}$.
	Note that due to the algebra property of $\mathcal A(U_m)$ the matrix $C^\tT C$ is indeed in $\mathcal{V}(U_m)^{m_2,m_2}$.
	
	Then, it holds $\mathcal M=F^{-1}(0)$.
	The differential $DF(C)\colon \mathcal T_C \mathcal V_1\to \mathcal T_{F(C)} \mathcal V_2$ of this mapping at $C \in \mathcal{V}(U_m)^{m_1,m_2}$ applied to $\Xi \in \mathcal T_C \mathcal V_1 = \mathcal{V}(U_m)^{m_1,m_2}$ is given by
	\begin{equation}\label{++}
	DF(C)[\Xi] = C^\tT \Xi + \Xi^\tT C.
	\end{equation}
	For every $C\in \mathcal M$ and  $V \in \mathcal V_2$, we obtain
	\[ 
	DF(C)[\tfrac12 C V] = \tfrac12 C^\tT C V + \tfrac12 V^\tT C^\tT C = V,
	\]
	resulting in $\text{range}(DF(C)) = \mathcal V_2$, i.e., that $DF(C)$ has full rank.
	Hence, Lemma \ref{lem:01} implies that $\mathcal M$ is a manifold and  $\mathcal T_C \mathcal M=\ker (DF(C))$. 
	We show that $\ker (DF(C)) = \mathcal S$, where
	$$\mathcal S \coloneqq \bigl\{C V + B: V \in \mathcal{V}(U_m)^{m_2,m_2}, B \in \mathcal{V}(U_m)^{m_1,m_2} \text{ with } V^\tT = -V, \, C^\tT B = 0\bigr\}.$$ 
	It follows directly from
	\begin{align*}
		C^\tT(CV + B) + (V^\tT C^\tT + B^\tT)C = V + V^\tT = 0
	\end{align*}
	that $\mathcal S \subseteq \ker(DF(C))$.
	To prove the opposite inclusion, assume that  $\Xi\in\ker(DF(C))$. Let $V\coloneqq C^\tT\Xi$ and $B\coloneqq \Xi-CC^\tT\Xi$.
	Then, we have $\Xi= CV+B$, where 
	\begin{align}
		&V+V^\tT =  C^\tT\Xi + \Xi^\tT C =0 \quad\text{and}\quad C^\tT B = C^\tT \Xi - C^\tT CC^\tT\Xi=0,
	\end{align}
	i.e., $\Xi \in \mathcal S$.
	
	Clearly, $\mathcal{S} \subseteq \mathcal{V}(U_m)^{m_1,m_2} \cap T_C \mathrm{St}(d,n)$.
	In order to verify the other inclusion, 
	consider $X  \in \mathcal{V}(U_m)^{m_1,m_2} \cap T_C \mathrm{St}(d,n)$, 
	where $C \in \mathcal M$.
	Then, there exist $V$ and $B$ such that $V^\tT = -V, \, C^\tT B = 0$ and
	$X = C V + B$.
	Multiplying the last equation by $C^\tT$ yields
	$C^\tT X = V$. 
	Due to the algebra structure of $\mathcal{A}(U_m)$, we obtain that $V \in \mathcal{V}(U_m)^{m_2,m_2}$ 
	and further $B \in \mathcal{V}(U_m)^{m_1,m_2}$.
		
	2. To show that the  projection and retraction are determined by those of the Stiefel manifold,
	it remains to show that the later ones map $\mathcal{V}(U_m) ^{m_1,m_2}$ into itself.
	Fortunately, this follows immediately from the definitions of the projection 
	in \eqref{proj_stiefel_1} and the retraction in \eqref{retraction_2} 
	and the algebra property of $\mathcal{A}(U_m)$.
	\end{proof}
	
	By Proposition \ref{prop:submanifold}, cPNNs with full filters can be trained using the same stochastic gradient descent Algorithm~\ref{alg:SGD_Stiefel} as for usual PNNs -- we solely need to stay within $\mathcal{V}(U_m) ^{m_1,m_2}$, i.e., minimize over the filters.
	Finally, we have the following proposition concerning orthogonal projections 
	with respect to the Frobenius norm
	onto our submanifold.
	
\begin{proposition}\label{prop:orth}
Let $X \in \mathcal{V}(U_m)^{m_1,m_2}$ have  full column rank. 
The orthogonal projection of a matrix $X \in \mathcal{V}(U_m)^{m_1,m_2}$ with full column rank
onto $\mathcal{V}(U_m)^{m_1,m_2} \cap \St(d,n)$ is uniquely determined by the matrix $U$ 
in the polar decomposition $X=U S$.
\end{proposition}                               

\begin{proof}
As $X$ has full rank, a unique polar decomposition of $X$ exists.
In view of \eqref{polar}, it suffices to show that $X \in \mathcal{V}(U_m)^{m_1,m_2}$ implies $U \in \mathcal{V}(U_m)^{m_1,m_2}$. 
To this end, we first decompose the matrix $S$.
Due to the algebra structure, we get $S^2 = X^* X \in \mathcal{V}(U_m)^{m_2,m_2}$.
The special structure of $\mathcal{V}(U_m)^{m_2,m_2}$ allows us to find a permutation matrix $P$ such that
\[S^2 = \blkdiag(\underbrace{U_m,\ldots,U_m}_\text{$m_2$ blocks}) P \blkdiag(A_1,\ldots,A_m) P^\tT \blkdiag(\underbrace{U_m^*,\ldots,U_m^*}_\text{$m_2$ blocks}),\]
where all $A_i \in \R^{m_2,m_2}$ are symmetric positive definite and $\blkdiag(A_1,\ldots,A_m)$ denotes the block diagonal matrix with blocks $A_1, \ldots, A_m$.
Clearly, the invertibility of $S^2$ implies that all $A_i$ are invertible.
Therefore, $S^{-1}$ can be written as
\[S^{-1} = \blkdiag(U_m,\ldots,U_m) P \blkdiag\bigl(A_1^{-1/2},\ldots,A_m^{-1/2}\bigr) P^\tT \blkdiag(U_m^*,\ldots U_m^*) \in \mathcal{V}(U_m)^{m_2,m_2}.\]
Hence, we obtain $U = X S^{-1} \in \mathcal{V}(U_m)^{m_2,m_2}$, which concludes the proof.
\end{proof}

\begin{example}
	From the previous lemma, we can directly conclude that the (not necessarily unique) 
	orthogonal projection of any matrix $\Circ_m(a)$, $a \in \mathbb R^m$, onto 
	$\mathcal A (F_m) \cap \St(m,m)$
	is explicitly given by $\Circ_m (a^*)$, where
	$$
	a^* =  F_m^* \tilde a , \quad 
	\tilde a_j = 
	\tfrac{\hat a_j}{|\hat a_j|}, \quad  
	\mathrm{if} \quad |\hat a_j| \not = 0, \quad \hat a \coloneqq \sqrt{m} F_m a
	$$
	and 
	$\tilde a_j = \overline{\tilde a}_{m-j} = \mathrm{e}^{2\pi i \phi_j}$  for any $\phi \in [0,1)$
	if $|\hat a_j| = |\hat a_{m-j}| = 0$.
\end{example}

\begin{remark}
	In case that $m_1 = m_2$, we can use the iteration
	\begin{equation}
		W_0 = X \quad \text{and} \quad W_{k+1} = \tfrac12 \bigl( W_k + (W_k^*)^{-1} \bigr)
	\end{equation}
	for computing $W$, see \cite{BX08,High86}.
	Note that these iteration preserves the algebra structure.
	Quadratic convergence of the iterates was proven for non singular matrices $X$.
	This could be used for an alternative proof of Proposition~\ref{prop:orth} for quadratic matrices.
	
	A more detailed review on polar decompositions can be found in \cite[Chap.~8]{Higham08}.
	In particular, for rectangular matrices the Newton--Schulz iteration
	\begin{equation}
		W_0 = X \quad \text{and} \quad W_{k+1} = \tfrac12 W_{k}(3I + W_k^* W_k) 
	\end{equation}
	was proposed, which convergences if the singular values of $X$ fulfill $0< \sigma_i<\sqrt{3}$.
\end{remark}

\subsection{Convolutional PNNs with Limited Filter Lengths}\label{sec:cPNN+limit}
Unfortunately, the approach from the previous section is not applicable 
to matrices arising from filters of small length $l < m$, since these
matrices do not form an algebra. 
In this section, we propose another approach for learning cPNNs
with filters of limited lengths.
First, we provide a justification for restricting our attention to circulant matrices.
More precisely, we show that for every  matrix with $m \times m$ Toeplitz blocks of filter lengths $l \le (m-4)/5$ 
lying in a Stiefel manifold, the corresponding circulant matrix is in the Stiefel manifold as well.

\begin{proposition}\label{thm:toep_orth_yields_circ_orth}
For $j=1,\ldots,m_1$, $k=1,\ldots,m_2$, $m_1 \geq m_2$ and  filter length $l \in \N$ with $l \le (m-4)/5$,
let 
$a^{(j,k)} = (a^{(j,k)}_0, a^{(j,k)}_{-1}, \ldots,a^{(j,k)}_{-l},0,\ldots,0,a^{(j,k)}_l,\ldots,a^{(j,k)}_1)^\tT$.
We consider block matrices with circulant blocks
\begin{equation}\label{blockcirc}
C=\left(\begin{array}{ccc}
\Circ_m(a^{(1,1)})&\cdots&\Circ_m(a^{(1,m_2)})\\
\vdots&&\vdots\\
\Circ_m(a^{(m_1,1)})&\cdots&\Circ_m(a^{(m_1,m_2)})
\end{array}\right)
\end{equation}
and block matrices with Toeplitz blocks
$$
T=\left(\begin{array}{ccc}
\Toep_m(a^{(1,1)})&\cdots&\Toep_m(a^{(1,m_2)})\\
\vdots&&\vdots\\
\Toep_m(a^{(m_1,1)})&\cdots&\Toep_m(a^{(m_1,m_2)})
\end{array}\right),
$$
where 
$\Toep_m(a^{(j,k)})$ denotes the $m\times m$ Toeplitz matrix with 
first row and column
$$\bigl(a_0^{(j,k)},a_1^{(j,k)},\cdots,a_l^{(j,k)},0,\cdots,0\bigr) 
\quad \mathrm{and} \quad
\bigl(a_0^{(j,k)},a_{-1}^{(j,k)},\cdots,a_{-l}^{(j,k)},0,\cdots,0\bigr)^\tT,
$$
respectively.
Then $T \in \St(d,n)$ with $d=m_2 m$, $n = m_1m$  implies that $C \in \St(d,n)$.
\end{proposition}

\begin{proof}
Let $T_i$ and $C_i$, $i=1,\ldots,d$, denote the columns of $T$ and $C$, respectively. 
Since $T\in\St(d,n)$, we have 
$$
\langle T_{i_1},T_{i_2}\rangle
=
\begin{cases}
1,&\mathrm{if} \; i_{1}=i_{2},\\
0,&\mathrm{otherwise}.
\end{cases}
$$
Further, the special structure of circulant and Toeplitz matrices 
implies $T_{i}=C_{i}$ for $i \in I \coloneqq \{jm+k: j=0,\ldots,m_2-1, k=l+1,\ldots,m-l\}$. 
Consequently, we obtain for all $i_1,i_2 \in I$  that 
\begin{align}
\langle C_{i_1},C_{i_2}\rangle
=
\begin{cases}
1,&\mathrm{if} \;  i_1=i_2,\\
0,&\mathrm{otherwise}.
\end{cases}\label{eq:orth1}
\end{align}
As $C$ has circulant blocks, we conclude for $k_1,k_1',k_2,k_2'\in\{1,\ldots,m\}$ that 
\begin{align}
\langle C_{j_1 m+k_1},C_{j_2 m+k_2}\rangle
=
\langle C_{j_1 m+k_1'},C_{j_2 m+k_2'}\rangle
\quad \text{if}\quad 
k_1-k_1'=k_2-k_2'\,\, \mathrm{mod}\,\, m.\label{eq:circ_inner_shift}
\end{align}
Now, let $i_1=j_1m+k_1$ and $i_2=j_2m+k_2$. 
We want to show that the inner product of $C_{i_1}$ and $C_{i_2}$ 
is equal to $1$ if $i_1=i_2$ and $0$ otherwise. 

If $k_1,k_2\in\{l+1,\ldots,m-l\}$, then the claim follows by \eqref{eq:orth1}. 
Hence, it remains to consider the other cases.
Assume that $k_1\in\{1,\ldots,l\}$. 
If $k_2\in\{1,\ldots,m-2l-2\}$, then 
we get by \eqref{eq:orth1}, \eqref{eq:circ_inner_shift}, $k_1+l+1 \leq 2l+1 < m-l$ and $k_2+l+1<m-l$ that
\begin{align}
\langle C_{i_1},C_{i_2}\rangle
=
\langle C_{i_1+l+1},C_{i_2+l+1}\rangle
=
\begin{cases}
1&\mathrm{if} \; i_1=i_2,\\
0&\mathrm{otherwise}.
\end{cases}
\end{align}
For $k_2\in\{m-2l-1,\ldots,m\}$, we see that $k_2-2l-1>l+1$ and $l+1<k_1+m-2l-1<m-l$. 
Consequently, we obtain by \eqref{eq:orth1} and \eqref{eq:circ_inner_shift} that
\begin{align}
\langle C_{i_1},C_{i_2}\rangle
=
\langle C_{i_1+m-2l-1},C_{i_2-2l-1}\rangle
=
\begin{cases}
1 & \mathrm{if} \; i_1=i_2,\\
0 & \mathrm{otherwise}.
\end{cases}
\end{align}
Note that the other cases work out analogously, which completes the proof.
\end{proof}

The following example shows that the reverse direction of Proposition~\ref{thm:toep_orth_yields_circ_orth} is not true.

\begin{example} Let $C= \Circ_m(a)$ with $a=(0,1,0\ldots,0)^\tT$, i.e., $C$ is a permutation matrix and therefore orthogonal. 
The corresponding Toeplitz matrix $T=\Toep_m(a)$ has first row zero and is consequently not orthogonal. 
\end{example}

Thus, the set of convolution filters (with length $l$) 
corresponding to a matrix in the Stiefel manifold arising from zero-padding is strictly included in the set of convolution filters corresponding to a matrix in the Stiefel manifold arising from periodic boundary conditions.
Analogous arguments can be applied to show similar results for circular matrices and other boundary conditions, e.g., mirror boundary.
Therefore, we restrict our attention to block matrices with circulant blocks obtained from filters of length
$2l+1$, i.e., matrices of the form \eqref{blockcirc}.
We denote such matrices by $\mathrm{bCirc}(l,m,m_1,m_2)$.
To learn the corresponding cPNN, we have to minimize over
\begin{equation}\label{challenge}
\mathcal{U}_2^l \coloneqq 
\mathrm{bCirc}\left(l,m,\tfrac{n_1}{m}, \tfrac{d}{m}\right)  
\times \ldots \times 
\mathrm{bCirc}\left(l,m,\tfrac{n_K}{m}, \tfrac{d}{m} \right) \cap \mathcal U_2.
\end{equation}
Unfortunately, $\mathcal{U}^l_2$ is no longer a submanifold of $\St(d,n)$.
Concerning the structure of block circulant matrices belonging to Stiefel manifolds, we have the following proposition.

\begin{proposition}\label{1d}
A block circulant matrix $\mathrm{bCirc}(l,m,m_1,m_2)$ as in \eqref{blockcirc} with $m_1 \ge m_2$ and filter length fulfilling $m \ge 4l+1$ belongs to $\St(d,n)$ with $d=m_2 m$, $n = m_1m$ if and only if the filters satisfy
$$
\sum_{t=1}^{m_1} \sum_{k=u-l}^{l} 
a_k^{(t,s_1)}a_{k-u}^{(t,s_2)} = \delta_{s_1,s_2} \delta_{u,0}, \quad u=0,\ldots,2l,\,s_1,s_2 = 1,\ldots,m_2.
$$
In particular, for $m_1 = m_2 = 1$, the filter $a^{(1,1)}$ is, up to the sign, a unit vector and the orthogonal projection of $\Circ_m(a)$ onto $\mathrm{bCirc}(l,m,1,1) \cap \St(m,m)$ is given by $\Circ_m(a^*)$, where 
$$
a^*=\mathrm{sgn}(a_j)e_j \quad\text{with}\quad j\in\argmax_{i=-l,\ldots,l} a_i.
$$
\end{proposition}                               

\begin{proof}
Using  \eqref{circs}, we obtain by straightforward computation that 
	$$
	\mathrm{bCirc}\left(l,m,m_1,m_2 \right)^\tT \mathrm{bCirc}\left(l,m,m_1,m_2 \right) = I_{m m_2,m m_2}
	$$
	is equivalent to
	$$
	\sum_{t=1}^{m_1} \sum_{k=-l}^l \sum_{r=-l}^l a_k^{(t,s_1)}a_r^{(t,s_2)} e^{2\pi i (k-r)j/m} = \delta_{s_1,s_2}
	$$
	for all $s_1,s_2 = 1,\ldots,m_2$ and all $j=0,\ldots,m-1$.
	Using
	\[\alpha_u^{(s_1,s_2)} \coloneqq \sum_{t=1}^{m_1} \sum_{k=\max(u-l,-l)}^{\min(l,u+l)} 
	a_k^{(t,s_1)}a_{k-u}^{(t,s_2)}, \quad u=-2l,\ldots,2l, \]
	these equations can be rewritten as
	\smash{$\sum_{u=-2l}^{2l} \alpha_u^{(s_1,s_2)} e^{2\pi i u j/m} = \delta_{s_1,s_2}$}.
	Since $m \ge 4l+1$, we can apply for fixed $s_1,s_2$ the inverse Fourier transform and obtain that 
	the filters must satisfy
	$$
	\alpha_u^{(s_1,s_2)} = \sum_{t=1}^{m_1} \sum_{k=u-l}^{l} 
	a_k^{(t,s_1)}a_{k-u}^{(t,s_2)} = \delta_{s_1,s_2} \delta_{u,0}, \quad u=0,\ldots,2l,\,s_1,s_2 = 1,\ldots,m_2.
	$$
	For $m_1 = m_2 = 1$ this simplifies to
	$$
		\alpha_u^{(1,1)} = \sum_{k=u-l}^l 
	a_k^{(1,1)}a_{k-u}^{(1,1)} = \delta_{u,0}, \quad u=0,\ldots,2l.
	$$
	Let $l_-=\min\{j\in\{-2l,...,2l\}:a^{(1,1)}_j\neq0\}$ and $l_+=\max\{j\in\{-2l,...,2l\}:a^{(1,1)}_j\neq0\}$.
	Then, it holds $l_+=l_-$ as $l_+>l_-$ implies 
	$$
	0=\alpha_{l_+-l_-}^{(1,1)}=a_{l_+}^{(1,1)}a_{l_-}^{(1,1)}\neq 0.
	$$
	Hence, we can conclude by $\alpha_0^{(1,1)}=1$ that the filter $a^{(1,1)}$ is up to the sign a unit vector.
	Finally, the orthogonal projection follows immediately.
\end{proof}

For limited filter length, we have to solve instead of \eqref{eq:opt_problem} the minimization problem
\begin{equation}\label{conv_NN}
\argmin_{(T,b,\alpha) \in \mathcal{U}_2^l \times \mathcal{U}_1 \times \mathcal{U}_0}
\frac1N\sum_{i=1}^N\ell(\mathbf{\Phi}(x_i;u),y_i)
=
\argmin_{u} 
\bigl\{ H(u) + \iota_{\mathcal{U}^l_2} (T) + \iota_{\mathcal{U}_0} (\alpha) \bigr\}.
\end{equation} 
Unfortunately, it seems difficult to fulfill the constraint that the matrices of the cPNN are in $\mathcal{U}^l_2$, i.e., to project onto
\begin{equation} \label{leider}
\mathrm{bCirc}(l,m,m_1,m_2) \cap \St(d,n), \quad n = m_1m, d=m_2m.
\end{equation}
The key problem is that projecting onto Stiefel manifolds, which can be done via the polar decomposition of matrices \cite[Sect.~7.3, 7.4]{HS2013}, increases the filter length.
Further, we have seen in Proposition~\ref{1d} that the optimization domain is not necessarily connected any more, making gradient based approaches likely to fail.
Hence, we propose to minimize the functional
\begin{align}\label{conv_denoiser2}
\argmin_{(T,b,\alpha)\in\mathrm{bCirc}(l,m,m_1,m_2)^K\times \mathcal{U}_1\times\mathcal{U}_0} \Bigl\{H(u)+\mu \sum_{k=1}^K 
\|T_k^\tT T_k-I_d\|_F^2+\iota_{\mathcal{U}_0}(\alpha)\Bigr\}, \quad \mu > 0
\end{align}
instead of \eqref{conv_NN}.
Numerical approximations can be efficiently computed using stochastic gradient descent with additional momentum variables to cope with local minima, such as the Adam optimizer \cite{KB2014}. 
Having a solution $(\tilde T,\hat b, \hat \alpha)$ of \eqref{conv_denoiser2}, we finally modify the matrices 
$\tilde T_k \in \mathrm{bCirc}(l,m,\frac{n_k}{m},\frac{d}{m})$ to be in \eqref{leider} by solving for each 
$k  \in \{1,\ldots,K\}$ the minimization problem
\begin{equation} \label{eq:approx}
\argmin_{T\in \mathrm{bCirc}(l,m,\frac{n_k}{m},\frac{d}{m} ) } F_\lambda(T) ,
\quad  
F_\lambda(T) \coloneqq \|T-\tilde T_k\|_F^2 + \lambda\|T^\tT T-I_n\|_F^2, \quad \lambda \gg 1.
\end{equation}
Note that the nonconvex functional $F_\lambda$ is continuous and coercive, so that there exists a global minimizer.
The following proposition establishes a relation between \eqref{eq:approx} and the orthogonal projection of $\tilde T_k$ onto \eqref{leider}, justifying our last step.
To this end, recall that  a sequence $\{F_j\}_{j\in\N}$ 
of functionals $F_j\colon \mathbb R^d \rightarrow (-\infty,+\infty]$ 
is said to $\Gamma$-\emph{converge} to $F \colon  \mathbb R^d \rightarrow (-\infty,+\infty]$ 
if the following two conditions are fulfilled for every $x \in \mathbb R^d$, see~\cite{Braides02}:
\begin{enumerate}
	\item[i)] $F(x) \leq \liminf_{j \rightarrow \infty} F_j(x_j)$ whenever  $x_j \to x$, 
	\item[ii)] there is a sequence $\{y_j\}_{n\in\N}$ with $y_j \to x$ and $\limsup_{j \to \infty} F_j(y_j) \le F(x)$.
\end{enumerate}
The importance of $\Gamma$-convergence relies in the fact that 
every cluster point of minimizers of $\{F_j\}_{j\in\N}$ is a minimizer of $F$.

\begin{proposition}\label{lem:gam_conv}
Let $C \in \mathrm{bCirc}(l,m,m_1,m_2)$ be given 
and let $(\lambda_j)_j$ be a sequence of positive numbers with $\lambda_j\to\infty$ as $j\to\infty$. 
Then, the sequence of functionals
\[F_{\lambda_j}(T) = \|T-C\|_F^2 + \lambda_j\|T^\tT T-I_n\|_F^2\]
$\Gamma$-converges to $F(T)\coloneqq \|T-C\|_F^2+\iota_{\St(d,n)}(T)$ on  $\mathrm{bCirc}(l,m,m_1,m_2)$ as $j\to\infty$.
\end{proposition}

\begin{proof} 
\emph{1.~$\liminf$-inequality:} We show that for every sequence 
$\{T_j\}_j$ in $\mathrm{bCirc}(l,m,m_1,m_2)$
converging to 
$T \in \mathrm{bCirc}(l,m,m_1,m_2)$ 
it holds
$$
F(T) \le \liminf_{j\to\infty}F_{\lambda_j}(T_j). 
$$
To this end, we distinguish two cases.
If $T\in\St(d,n)$, then 
$$
F(T)=\|T-C\|_F^2 = \lim_{j\to\infty} \|T_j - C\|_F^2 \leq \liminf_{j \to \infty} F_{\lambda_j}(T_j).
$$
For $T\not\in\St(d,n)$, we get using $T_j\to T$ that $\|T_j^\tT T_j-I_n\|_F^2$ 
is bounded from below by some constant $c>0$ for $j$ sufficiently large. 
Thus, we obtain
$$
\liminf_{j \to \infty} F_{\lambda_j}(T_j) \geq \liminf_{j\to\infty} \lambda_j\, c = \infty = F(T).
$$
\emph{2.~$\limsup$-inequality:} We have that $F_{\lambda_j}$ converges pointwise to $F$.
Consequently, we get for the sequence $\{T_j\}_j$ with $T_j = T$ that 
$
\lim_{j\to\infty} F_{\lambda_j} (T) = F(T).
$
\end{proof}

The continuous, coercive function $\| \cdot - C\|^2$ is a lower bound for all functions $F_j$.
Thus, the functions $F_j$, $j \in \mathbb N$, are equi-coercive.
Together with this property, Proposition \ref{lem:gam_conv} yields that every cluster point of the minimizers of $F_{\lambda_j}$ 
is a  minimizer of $F$, see \cite{Braides02}.

A gradient descent scheme for solving \eqref{eq:approx}  is outlined in Algorithm~\ref{alg:Proj_Stiefel}.
Here, $\rho^{(r)}$ is an estimate of the local Lipschitz constant.
Note that the involved derivatives can be computed easily using two times the reverse mode of algorithmic differentiation, see \cite{HS2020}.

\begin{algorithm}[!t]
	\begin{algorithmic}
		\State \textbf{Initialization:}  $T^{(0)}\in\R^{d,n}$
		
		\For{$r=0,1,\ldots$}
		\State $g^{(r)} =\nabla_T F_\lambda(T^{(r)})/\Vert \nabla_T F_\lambda(T^{(r)})\Vert$
		\State $\rho^{(r)}=\|\nabla_T^2 F_\lambda(T^{(r)})\, g^{(r)}\|$
		\State $T^{(r+1)}=T^{(r)}-\nabla_T F_\lambda(T^{(r)})/\rho^{(r)}$
		\EndFor
	\end{algorithmic}
	\caption{Gradient descent scheme for solving \eqref{eq:approx} 
	} \label{alg:Proj_Stiefel}
\end{algorithm}

\begin{remark}[2D Convolutions]
So far we were concerned with convolution matrices $C$ for signals.
When switching to 2-dimensional structures, say images $X$,
we convolve row and columnwise as $C_1 \, X \, C_2^\tT$. 
Reshaping the matrices columnwise $X \to \mathrm{vec} (X)$, we can reformulate the above convolution as
$$
\mathrm{vec} (C_1 \, X \, C_2^\tT) = (C_2 \otimes C_1) \mathrm{vec} (X),
$$
where the Kronecker product $C_2 \otimes C_1$ is now a block circulant matrix with circulant blocks.
Replacing usual circulant matrices with these matrices does not affect the conclusions drawn in this section.
\end{remark}

\section{Scaled cPNNs as Denoisers}\label{sec:DenoisePnNN}
%
Although cPNNs may be useful in different contexts, we focus on their application for denoising of signals and images.
For this purpose, let $y_i\in\R^{m}$, $i=1,\ldots,N$, be ground truth signals with
noisy versions $x_i=y_i+\epsilon_i$ corrupted by additive Gaussion noise with realizations
$\epsilon_i\sim\mathcal N(0,\sigma^2)$.
In Subsection~\ref{subsec:denoise}, we show how to train cPNNs for denoising.
In particular, there are two important observations:
\begin{itemize}
\item[i)] Learning the noise leads to better results than learning to denoise directly.
This approach is known as residual learning in the literature \cite{ZZCMZ2017}.
\item[ii)] Learning scaled cPNNs with \emph{one additional fixed scaling parameter} $\gamma >1$ 
is superior to standard cPNNs.
Note that $\gamma$ is an upper bound on the Lipschitz constant of the whole network.
\end{itemize}
Subsection \ref{subsec:pwc_signals_images} demonstrates the performance of scaled cPNNs for denoising of signals and images corrupted by Gaussian noise based on numerical results.
In the subsequent Section~\ref{sec:PnP}, we apply cPNN based denoisers within a PnP framework.

\subsection{Learning scaled cPNNs}\label{subsec:denoise}
For inputs $x \in \R^m$, we train a cPNN $\Phi(\cdot\,;u)$  
with parameters $u=(T,b,\alpha)\in\mathcal{U}_2^l \times \mathcal{U}_1 \times  \mathcal{U}_0$ and identical layer sizes $T_k \in \R^{m_1 m,m_2 m}$, $k=1,\ldots,K$, where $1 \le m_1 \le m_2$, i.e.,
$n_k = m_1 m$ for all $k=1,\ldots,K$.
Moreover, the parameters in $\mathcal{U}_1$ are chosen in a simplified way, 
namely starting from a vector $b = (b_1,\ldots,b_{m_1})^\tT \in \R^{m_1}$, we use
\begin{equation} \label{vec_b}
\mathcal{U}_1 \coloneqq \bigl\{( b \otimes 1_m)^K: b \in \R^{m_1} \bigr\}.
\end{equation}
As usual, the whole network is trained with $m_2$ copies of (patches of) the signals/images.
In other words, we introduce $A \in \St(m,m_2 m)$ given by
$$
A =\frac1{\sqrt{m_2}}\left(\begin{array}{c} I_{m}\\\vdots\\I_{m}\end{array}\right)\in\R^{m_2 m,m}
$$
and define $\Psi\colon\R^{m}\to\R^{m}$ by
\begin{equation}\label{eq:DenoisePNN}
	\Psi (x;u)= A^\tT\Phi(A x;u).
\end{equation}
Clearly, $\Psi\colon \R^m \rightarrow \R^m$ is again an averaged operator.
For a fixed scaling parameter $\gamma \ge 1$,
we aim to optimize the network parameters $u$ such that 
$\gamma \Psi(x_i;u) \approx \epsilon_i=x_i-y_i$, $i=1,\ldots,N$. 
For this purpose, we intend to solve
\begin{align}
\argmin_{(T,b,\alpha)\in \mathcal{U}_2^l \times \mathcal{U}_1 \times \mathcal{U}_0}
\frac1N\sum_{i=1}^N \ell\bigl(\gamma \Psi(x_i;u),\epsilon_i\bigr)
=
\argmin_u \bigl\{H(u)+\iota_{\mathcal{U}_2^l}(T)+\iota_{\mathcal{U}_0}(\alpha)\bigr\} \label{conv_denoiser1_b}
\end{align}
with $H(u) \coloneqq \frac1N\sum_{i=1}^N \ell(\gamma \Psi(x_i;u),\epsilon_i)$ 
and the quadratic loss function $\ell(x,y) = \|x-y\|^2$.

For cPNNs with filters of full length, we apply the stochastic gradient descent Algorithm~\ref{alg:SGD_Stiefel}
on the corresponding submanifold of the Stiefel manifold as shown in Subsection~\ref{sec:cPNN+matrix}. 
However, this is only possible with reasonable effort for signals of moderate size, such as the examples in the next subsection.
For cPNNs with limited filter length, we use the procedure from Subsection \ref{sec:cPNN+limit}, consisting of the two steps outlined in Algorithm~\ref{alg:cPNN_limited}.

\begin{algorithm}[t]
	\begin{algorithmic}
		\State \textbf{Input:}  Training data $x_i,\varepsilon_i$, $i=1,\ldots,N$; parameters $\gamma \ge 1$, $\mu >0$, $\lambda \gg 1$
		\State \textbf{Computation:}
		\State 1. Solve the relaxed  problem
\begin{align}
(\tilde T, \hat b, \hat \alpha) \in \argmin_{(T,b,\alpha)\in\mathrm{bCirc}(l,m,m_1,m_2)^K\times \mathcal{U}_1\times\mathcal{U}_0} 
\Bigl\{H(u)+\mu \sum_{k=1}^K \|T_k^\tT T_k-I\|_F^2+\iota_{\mathcal{U}_0}(\alpha)\Bigr\},\label{conv_denoiser2_b}
\end{align}
with the Adam optimizer \cite{KB2014}.
\vspace{.1cm}

\State 2. Compute the solution $\hat T$ of \eqref{conv_denoiser2} via Algorithm \ref{alg:Proj_Stiefel} to approximate the projection of $\tilde T$ onto $\mathcal{U}_2^l$.
\end{algorithmic}
\caption{Training scaled cPNNs with limited filter length}
\label{alg:cPNN_limited}
\end{algorithm}

Usually the network is trained only with patches of the original signals/images and the signal/image that we want to denoise is of larger size $\tilde m > m$.
To construct a denoiser from the trained network, we just fill in  
the convolution filters in $\hat T_k$, $k=1,\ldots,K$, by zeros to get matrices in $\mathrm{bCirc}(l,\tilde m,m_1,m_2)$, 
and the vectors of $\mathcal U_1$ in \eqref{vec_b} by ones, i.e., we use $\hat b \otimes 1_{\tilde m}$.
Clearly, the resulting network $\Psi\colon \mathbb R^{\tilde m} \to \mathbb R^{\tilde m}$ is again averaged
and we use the same notation $\Psi$ as the size is obvious from the context.
Finally, we compute the denoised version $y\in\R^{\tilde m}$ of $x\in\R^{\tilde m}$ by applying the denoiser
\begin{equation} \label{eq:denoiser}
y = x - \gamma \Psi(x;\hat u)  = \mathcal D (x), \quad \mathcal D \coloneqq I_{\tilde m} - \gamma \Psi.
\end{equation}

\subsection{Numerical Results for Denoising }\label{subsec:pwc_signals_images}
In this subsection, we present denoising results for signals and images, where we use the 
ReLU activation function $\sigma(x)=\max(x,0)$.
The quality of the results is compared in terms of the peak signal-to-noise ratio (PSNR), which is defined for a predicted signal $x\in\R^d$ and a ground truth $y\in\R^d$ as
$$
\mathrm{PSNR}(x,y)\coloneqq 10\log_{10}\biggl(\frac{\max y-\min y}{\tfrac1d\sum_{i=1}^d (x_i-y_i)^2}\biggr).
$$
If we deal with gray-valued images living on $[0,1]^{d_1\times d_2}$ instead, the PSNR is defined as
$$
\mathrm{PSNR}(x,y)\coloneqq 10\log_{10}\biggl(\frac{1}{\tfrac1{d_1d_2}\sum_{i=1}^{d_1}\sum_{j=1}^{d_2} (x_{ij}-y_{ij})^2}\biggr).
$$

\paragraph{Denoising of piecewise constant signals}
First, we apply scaled cPNNs for denoising the same signals as in \cite{HHNPSS2019}. 
Note  that in \cite{HHNPSS2019} the PNNs were neither convolutional nor using residual learning.
By $(x_i,y_i)\in\R^{m}\times\R^{m}$, $i=1,\ldots,N$, we denote pairs of piecewise constant signals $y_i$ of length $m=128$ and their noisy versions by $x_i=y_i+\epsilon_i$, 
where $\epsilon_i$ is Gaussian noise with standard deviation $\sigma = 0.1$. 
For the \emph{signal generation}, we choose 
\begin{itemize}
\item  the number of constant parts of $y_i$ as $\max\{2,t_i\}$, where $t_i$ is the realization of a random variable 
following the Poisson distribution with mean $5$;
\item the discontinuities of $y_i$ as realization of a uniform distribution;
\item the signal intensity of $y_i$ for every constant part as realization of the standard normal distribution, where we subtract the mean of the signal finally.
\end{itemize}
Using this procedure, we generate training data $(x_i,y_i)_{i=1}^N$ and test data $(x_i,y_i)_{i=N+1}^{N+N_\text{test}}$ 
with $N=500000$ and $N_\text{test}=1000$. Note that $\tilde m = m$ here.
Further, the average PSNR of the noisy signals in the test set is $25.22$.

For denoising, we apply scaled cPNNs with $K=5$ layers, $m_1=64$, $m_2=128$
and different scaling factors $\gamma\in\{1,1.99,5\}$ both for filters of full and limited length $2l+1$ with $l\in\{2,5,10\}$.
We train the networks using the approximative Algorithm \ref{alg:cPNN_limited}.
As loss function we use the squared $\ell^2$-norm $\ell(x,y)=\|x-y\|_2^2$.
Moreover, we include results obtained by a CNN with the same architecture as our cPNNs, but without orthogonality constraint on the rows of the matrices $T_k$, $k=1,\ldots,K$.
Table \ref{tab:full_vs_limit} contains the resulting average PSNRs on the test set.
As expected, the PSNRs increase for larger filter lengths.
Further, we observe that the results with limited filters become better than the ones with full filters if $l$ is large enough. 
This is most likely caused by overfitting effects, e.g., for $\gamma=5$ the loss function on the training set is given by $0.0734$ for full filters and by $0.0747$ for limited length $l=10$. 
That is, although the average PSNR on the test set is better for limited filter filter lengths, the loss on the training set is worse than for full filters.
As we do not observe overfitting for limited filter lengths, further discussion of this issue is out of scope of this paper.

\begin{table}
\centering
\begin{tabular}{c|cccc}
Method&full filters&$l=2$&$l=5$&$l=10$\\\hline
cPNN, $\gamma=1\phantom{.00}$&$35.73$&$34.24$&$35.03$&$35.52$\\
cPNN, $\gamma=1.99$&$37.93$&$36.12$&$36.88$&$37.68$\\
cPNN, $\gamma=5\phantom{.00}$&$38.88$&$37.24$&$37.99$&$39.02$\\
unconst.~CNN\phantom{.0}&$39.26$&$37.80$&$38.52$&$39.26$\\
\end{tabular}
\caption{Average PSNR for denoising piecewise constant signals 
corrupted by Gaussian noise with $\sigma=0.1$ using different cPNNs.}
\label{tab:full_vs_limit}
\end{table}

\begin{table}[p]
\centering
\begin{tabular}{c|c|c|c|c|c|c|c}
	Method & BM3D & $\gamma=1$ & $\gamma=1.99$ & $\gamma=5$ & $\gamma=10$ & unconst.~CNN& DnCNN\\\hline
	PSNR & $28.59$ & $28.48$ & $28.81$ & $29.02$ & $29.08$ & $29.11$ & $29.23$
\end{tabular}
\caption{Average PSNR for denoising images from the BSD68 test set corrupted by Gaussian noise with $\sigma = 25/255$ using the DnCNN, BM3D, cPNNs for different choices of $\gamma$ and an unconstrained CNN with the same structure as the cPNNs.}
\label{tab:results_gamma}
\end{table}

\begin{figure}[p]
\begin{subfigure}[t]{0.166\textwidth}
\includegraphics[width=\textwidth]{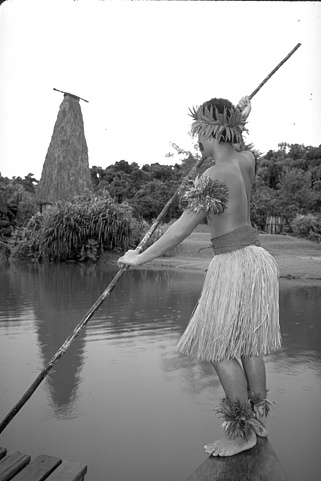}
\end{subfigure}\hfill
\begin{subfigure}[t]{0.166\textwidth}
\includegraphics[width=\textwidth]{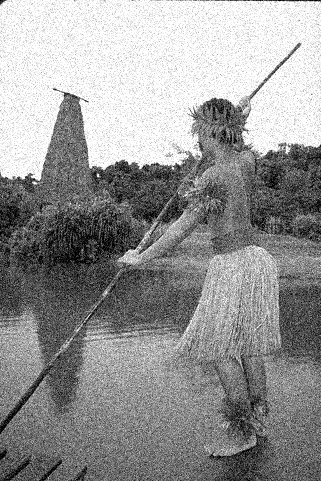}
\end{subfigure}\hfill
\begin{subfigure}[t]{0.166\textwidth}
\includegraphics[width=\textwidth]{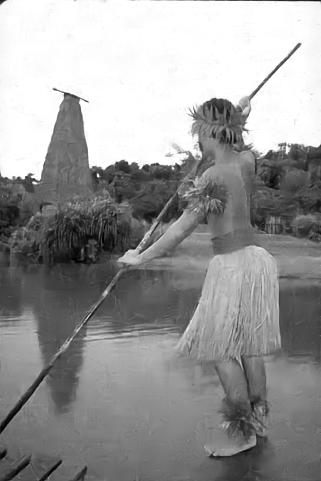}
\end{subfigure}\hfill
\begin{subfigure}[t]{0.166\textwidth}
\includegraphics[width=\textwidth]{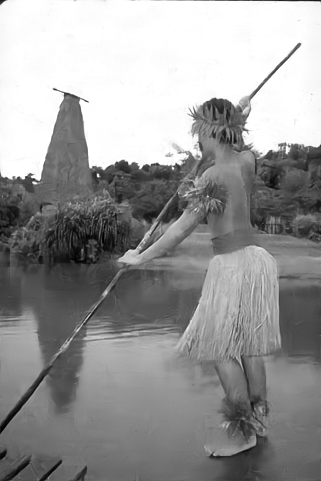}
\end{subfigure}\hfill
\begin{subfigure}[t]{0.166\textwidth}
\includegraphics[width=\textwidth]{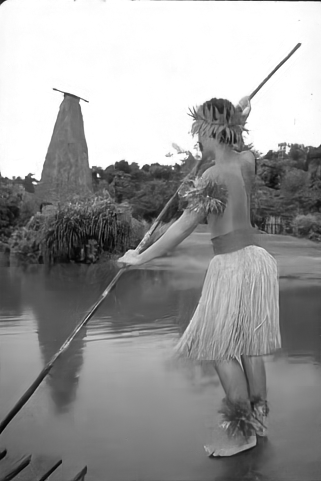}
\end{subfigure}\hfill
\begin{subfigure}[t]{0.166\textwidth}
\includegraphics[width=\textwidth]{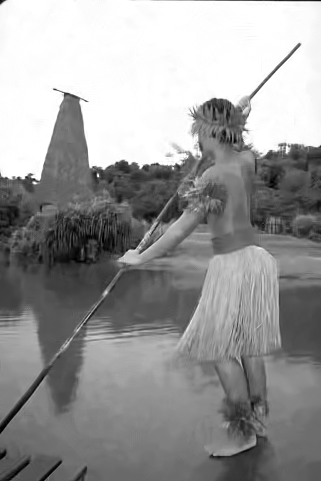}
\end{subfigure}
\begin{subfigure}[t]{0.166\textwidth}
\includegraphics[width=\textwidth]{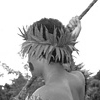}
\caption*{\scriptsize Original}
\end{subfigure}\hfill
\begin{subfigure}[t]{0.166\textwidth}
\includegraphics[width=\textwidth]{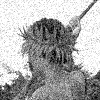}
\caption*{\scriptsize Noisy\\PSNR $20.15$}
\end{subfigure}\hfill
\begin{subfigure}[t]{0.166\textwidth}
\includegraphics[width=\textwidth]{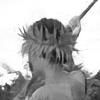}
\caption*{\scriptsize cPNN, $\gamma=1.99$\\PSNR $29.84$}
\end{subfigure}\hfill
\begin{subfigure}[t]{0.166\textwidth}
\includegraphics[width=\textwidth]{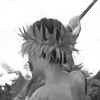}
\caption*{\scriptsize cPNN, $\gamma=5$\\PSNR $30.07$}
\end{subfigure}\hfill
\begin{subfigure}[t]{0.166\textwidth}
\includegraphics[width=\textwidth]{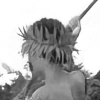}
\caption*{\scriptsize unconst.~CNN\\PSNR $30.19$}
\end{subfigure}\hfill
\begin{subfigure}[t]{0.166\textwidth}
\includegraphics[width=\textwidth]{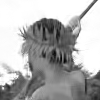}
\caption*{\scriptsize BM3D\\PSNR $29.62$}
\end{subfigure}
\vspace{.5cm}

\begin{subfigure}[t]{0.166\textwidth}
\includegraphics[width=\textwidth]{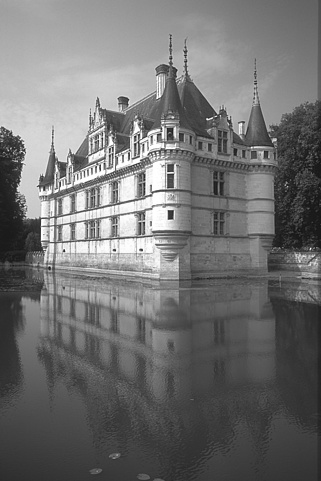}
\end{subfigure}\hfill
\begin{subfigure}[t]{0.166\textwidth}
\includegraphics[width=\textwidth]{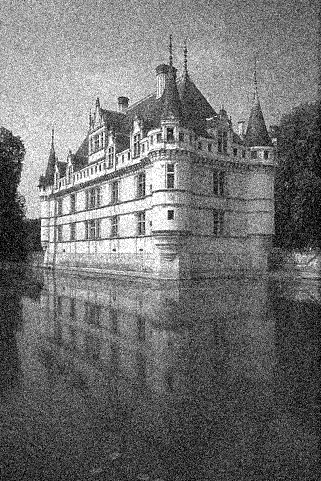}
\end{subfigure}\hfill
\begin{subfigure}[t]{0.166\textwidth}
\includegraphics[width=\textwidth]{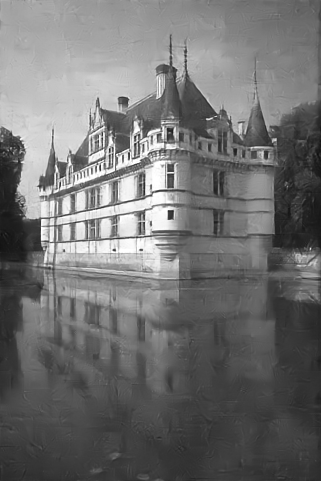}
\end{subfigure}\hfill
\begin{subfigure}[t]{0.166\textwidth}
\includegraphics[width=\textwidth]{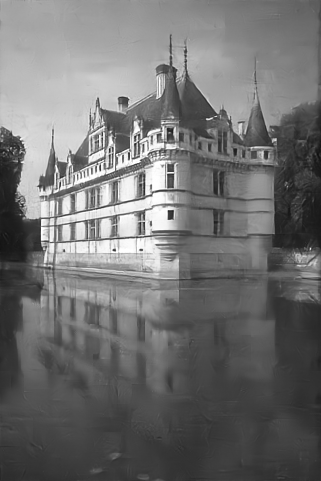}
\end{subfigure}\hfill
\begin{subfigure}[t]{0.166\textwidth}
\includegraphics[width=\textwidth]{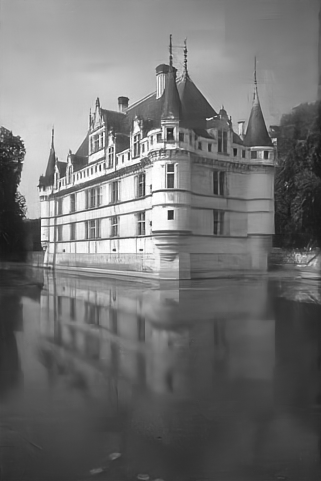}
\end{subfigure}\hfill
\begin{subfigure}[t]{0.166\textwidth}
\includegraphics[width=\textwidth]{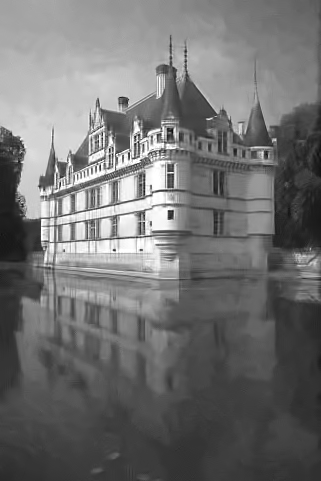}
\end{subfigure}
\begin{subfigure}[t]{0.166\textwidth}
\includegraphics[width=\textwidth]{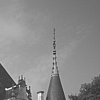}
\caption*{\scriptsize Original}
\end{subfigure}\hfill
\begin{subfigure}[t]{0.166\textwidth}
\includegraphics[width=\textwidth]{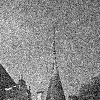}
\caption*{\scriptsize Noisy\\PSNR $20.17$}
\end{subfigure}\hfill
\begin{subfigure}[t]{0.166\textwidth}
\includegraphics[width=\textwidth]{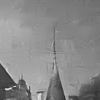}
\caption*{\scriptsize cPNN, $\gamma=1.99$\\PSNR $29.68$}
\end{subfigure}\hfill
\begin{subfigure}[t]{0.166\textwidth}
\includegraphics[width=\textwidth]{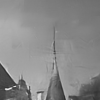}
\caption*{\scriptsize cPNN, $\gamma=5$\\PSNR $29.93$}
\end{subfigure}\hfill
\begin{subfigure}[t]{0.166\textwidth}
\includegraphics[width=\textwidth]{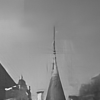}
\caption*{\scriptsize unconst.~CNN\\PSNR $29.86$}
\end{subfigure}\hfill
\begin{subfigure}[t]{0.166\textwidth}
\includegraphics[width=\textwidth]{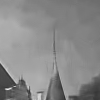}
\caption*{\scriptsize BM3D\\PSNR $29.56$}
\end{subfigure}
\caption{Denoising results for images corrupted by Gaussian noise with $\sigma = 25/255$ using different methods.}
\label{fig_denoise}
\end{figure}

\paragraph{Denoising of gray-valued images}

Next, we want to denoise natural images with gray values in $[0,1]$ 
corrupted by Gaussian noise with standard deviation 
$\sigma = 25/255 \approx 0.098$ using scaled cPNNs.
We train cPNNs with the parameters $m_1=64$, $m_2=128$, $K=8$, $l=5$ and $\gamma\in\{1,1.99,5,10\}$.
As training data, we use $60000$ patches of size $40\times 40$, which are cropped from the $400$  training and validation images in the BSDS500 dataset \cite{MFTM2001}.
Using the loss function $\ell(x,y)=\|x-y\|_F^2$, we train the network with the Adam optimizer for $3000$ epochs.

For assessing the quality of our scaled cPNNs, we apply them onto images from the BSD68 test set.
As comparison we use DnCNN \cite{ZZCMZ2017}, BM3D \cite{DFKE2007} 
and an unconstrained CNN with the same architecture as our cPNN.
The noisy images have an average PSNR of $20.17$ and the resulting average PSNRs for the different methods are given in Table~\ref{tab:results_gamma}.
Additionally, two example images are provided in Figure~\ref{fig_pnp_blur}.
We observe that for $\gamma=1.99$ our proposed cPNN outperforms BM3D.
This particular choice of $\gamma$ turns out to be important in the next section.
If we further increase $\gamma$, the average PSNR becomes even better.
For $\gamma=10$ the average PSNR of our cPNN is approximately $0.2$ below DnCNN.
Overall, we observe that the increase in PSNR saturates quite quickly, i.e., 
using a network with relatively low Lipschitz constant turns out to be sufficient.

\section{Plug-and-Play Algorithms with scaled cPNNs} \label{sec:PnP}
PnP algorithms were first introduced in \cite{SVW2016,VBW13} and have led to improved results in various image restoration 
problems, see \cite{CWE2016,MMHC2017,Ono2017,TBF2017}.
Such algorithms are based on the observation that the proximal operator with respect to the regularizer in the ADMM is basically a denoising step.
Hence, it seems natural to replace this step by a more powerful denoiser such as BM3D \cite{DFKE2007,DKE2012} or NNs \cite{SWK2019}.
Moreover, the approach can be also used within the ISTA algortihm \cite{DDM04}, where the proximal operator of the regularizer corresponds to a soft thresholding step.
However, soft thresholding is just the proximal function of the (grouped) $\ell_1$ norm and ISTA itself is a special case of forward-backward splitting algorithms (FBS) \cite{CW05}.
Consequently, the same argumentation holds for the FBS algorithm and its accelerations as well.
FBS and ADMM are designed for minimizing the sum
\begin{equation}\label{eq:sum}
	f(x) + g(x), \quad f,g \in \Gamma_0(\R^m)
\end{equation}
and outlined in Algorithms~\ref{alg:FBS} and \ref{alg:ADMM}, respectively.
For the FBS algorithm, the function $f$  has to be additionally differentiable with Lipschitz continuous gradient.
The modified PnP steps are included in the indented rows of the algorithms.
In general, the PnP variants do not minimize a sum of functions as in \eqref{eq:sum}.
To guarantee convergence of the algorithm, $\mathcal{D}$ has to fulfill certain
properties.
Indeed, it was shown in \cite{SKM2019} that we do not have convergence within a FBS-PnP  
for $\mathcal D = \text{DnCCN}$, see \cite{ZZCMZ2017} for details on DnCCN.

\begin{figure}[t]
\algcaption{FBS and FBS-PnP (intended row)}\label{alg:FBS}
\begin{algorithmic}
\State \textbf{Initialization:} $y^{(0)}\in \mathbb R^m$, $\eta \in (0,\frac{2}{L} )$
\State \textbf{Iterations:} For $r = 0,1,\ldots$
\State
$
\begin{array}{lcl}
y^{(r+1)} &=& x^{(r)} - \eta \nabla f (x^{(r)})
\\
x^{(r+1)} &=& \prox_{\eta g} ( y^{(r+1)})
\end{array}
$
\State \hspace{1cm}
\textbf{PnP Step:} 
$\mathbf{
\begin{array}{lcl}
x^{(r+1)} &=& \mathcal{D} ( y^{(r+1)})
\end{array}
}
$
\end{algorithmic}
\bottomrule
\vspace{.4cm}

\algcaption{ADMM and ADMM-PnP (intended row)}\label{alg:ADMM}
\begin{algorithmic}
\State \textbf{Initialization:} $y^{(0)}\in \mathbb R^m$, $p^{(0)} \in \mathbb R^m$, $\eta > 0$
\State \textbf{Iterations:} For $r = 0,1,\ldots$
\State
$
\begin{array}{lcl}
x^{(r+1)}&=& \prox_{\frac{1}{\eta} f} ( y^{(r)} - \tfrac{1}{\eta } p^{(r)} )
\\
y^{(r+1)} &=& \prox_{\frac{1}{\eta} g} (  x^{(r+1)} + \tfrac{1}{\eta} p^{(r)})
\end{array}
$
\State \hspace{1cm}
\textbf{PnP Step:} 
$
\begin{array}{lcl}
y^{(r+1)} &=& \mathcal{D} ( x^{(r+1)} + \tfrac{1}{\eta} p^{(r)})
\end{array}
$
\State
$
\begin{array}{lcl}
p^{(r+1)}&=& p^{(r)} + \eta( x^{(r+1)} -  y^{(r+1)} ) 
\end{array}
$
\end{algorithmic}
\bottomrule
\end{figure}

The following proposition summarizes convergence results for FBS-PnP and ADMM-PnP.
To make the paper self-contained and as we have not found an explicit reference,
we provide the proof in the appendix.
\begin{proposition}\label{conv_ADMM_PnP}
\hspace{.1em}
\begin{itemize}
	\item[i)] Let $f\colon\R^m\to\R$ be convex and differentiable with $L$-Lipschitz continuous gradient 
	and let $\mathcal{D}\colon\R^m\to\R^m$ be averaged. 
	Then, for any $0<\eta<\tfrac2L$, the sequence $\{x^{(r)} \}_r$ generated by the FBS-PnP algorithm converges.
	\item[ii)] Let $f\in \Gamma_0(\R^m)$
	and $\mathcal{D}\colon\mathbb R^m \rightarrow \mathbb R^m$ be $\frac12$-averaged.
	Then, the sequence $\{x^{(r)} \}_r$ generated by the \textrm{ADMM-PnP} converges.
	
\end{itemize}
\end{proposition}

Regarding the setting of part ii), we cannot hope for convergence if $\mathcal{D}$ is just averaged as the following example shows.

\begin{example}
In the following, we consider the case $m=1$ and provide for any $1>t>\tfrac12$ an $t$-averaged operator $\Psi$ and a function $f$ such that ADMM-PnP diverges.
For simplicity of notation, we assume $\gamma=1$ and remark that for any other $\gamma>0$ a similar example can be constructed. 
Let
\[\Psi(x)\coloneqq -a_1 x=\tfrac{1-a_1}{2}x+\tfrac{1+a_1}{2}(-x), \quad a_1\coloneqq 2t-1\in (0,1)\]
and \smash{$\prox_{f}=\tfrac12 (I+R)$}, where $R(x)=-a_2 x$ with $1>a_2>0$.
By Proposition~\ref{prop:stab}, we get that \smash{$\prox_f$} is indeed a proximity operator.
Now, assume that for $y^{(0)}\in\R\backslash\{0\}$ and $p^{(0)}=0$ the sequence $(x^{(r)},y^{(r)},p^{(r)})_r$ is generated by ADMM-PnP.
Then, we define $t^{(r)}\coloneqq x^{(r+1)}+p^{(r)}$ and use this to rewrite the iterations as
\begin{align}
&y^{(r+1)}=\Psi\bigl(t^{(r)}\bigr)=-a_1 t^{(r)}\label{y_on_t},\\
&p^{(r+1)}=t^{(r)}-y^{(r+1)}=(1+a_1)t^{(r)}.
\end{align}
Thus, we get that
\begin{align}
x^{(r+2)}=\prox_f\bigl(y^{(r+1)}-p^{(r+1)}\bigr)=\prox_f\bigl(-(1+2a_1)t^{(r)}\bigr)=-\tfrac{(1-a_2)(1+2a_1)}{2}t^{(r)},\label{x_on_t}
\end{align}
leading to the following recursion formula for $t$:
$$
t^{(r+1)}=x^{(r+2)}+p^{(r+1)}=\bigl(1+a_1-\tfrac{(1-a_2)(1+2a_1)}{2}\bigr)t^{(r)}.
$$
Now, choose $0<a_2<1$ large enough such that $a_1>\tfrac{(1-a_2)(1+2a_1)}{2}$. Then, we get that
$$
t^{(r+1)}=ct^{(r)},\quad c=(1+a_1)-\tfrac{(1-a_2)(1+2a_1)}{2}>1.
$$
Since $t^{(0)}=x^{(1)}+p^{(0)}=-a_2 y^{(0)}\neq 0$, this implies that the sequence $(t^{(r)})_r$ diverges.
By \eqref{y_on_t} and \eqref{x_on_t}, we obtain that the sequences $(x^{(r)})_r$ and $(y^{(r)})_r$ diverge.
\end{example}

\begin{remark}
Note that FBS-PnP (Algorithm~\ref{alg:FBS}) with data fidelity term $f$ and parameter $\eta$ is equivalent to the FBS-PnP iteration with data fidelity term $\eta f$ and parameter $1$.
Consequently, the parameter $\eta$ in FBS-PnP controls the weighting between denoising step and data fit, where a larger value of $\eta$ corresponds to a better fit. 
Thus, the convergence condition $\eta\in(0,\tfrac2L)$ limits the noise levels that can be tackled with FBS-PnP.

Similarly, in ADMM-PnP we replace the proximal operator $\prox_{\eta^{-1} g}$ with respect to the regularizer $g$ by our denoiser. 
Again, $\eta$ controls the weighting between denoising and data fit and a larger value of $\eta$ corresponds to a better fit.
Since ADMM-PnP converges independent of $\eta$, any noise level can be tackled.
\end{remark}

\paragraph{Choice of the denoiser $\mathcal{D}$ using scaled cPNNs}
Here, we want to apply our denoiser $\mathcal D = I - \gamma \Psi$
from \eqref{eq:denoiser} within the PnP framework.
Unfortunately, although $\Psi$ is averaged,  this is no longer true for $\mathcal D$.
As a remedy, we propose to use an oracle $x^*$ to obtain again an averaged operator as denoiser.

\begin{lemma}\label{lem:avaraged}
Let $x^* \in \mathbb R^{m}$ be fixed.
Further, let $\Psi\colon \mathbb R^{m} \to \mathbb R^{m}$ be an $t$-averaged operator 
with $t \in [\tfrac12,1]$.
For a scaling factor $0<\gamma<2$, the mapping
\begin{equation} \label{eq:oracle_denoise}
\mathcal D(x)=\big(1-\tfrac{1}{1-\gamma+2t\gamma}\big)x^*+\tfrac{1}{1-\gamma+2t\gamma}(I_m-\gamma \Psi(x))
\end{equation}
is $\tilde t$-averaged with $\tilde t=\tfrac{t\gamma}{1-\gamma+2t\gamma}$.
\end{lemma}

\begin{proof}
Since $\Psi$ is $t$-averaged, we have that $\Psi=(1-t)I_m+t R$ 
for some non-expansive operator $R$.
Thus, we get 
$$
I_m-\gamma \Psi=(1-\gamma+t\gamma) I_m -t\gamma R.
$$
As $0<\gamma<2$ and $t\geq\tfrac12$, we have that $1-\gamma+t\gamma\geq 1-\tfrac{\gamma}{2}>0$.
This implies 
$$
\tfrac{1}{1-\gamma+2t\gamma}(I_m-\gamma \Psi)
=
\tfrac{1-\gamma+t\gamma}{1-\gamma+2t\gamma}I_m-\tfrac{t\gamma}{1-\gamma+2t\gamma} R
=
\big(1-\tfrac{t\gamma}{1-\gamma+2t\gamma}\big)I_m-\tfrac{t\gamma}{1-\gamma+2t\gamma}R.
$$
Since $(1-\tfrac{1}{1-\gamma+2t\gamma})x^*$ is a constant, this proves the claim. 
\end{proof}

In our numerical examples, we observe that $\Psi$ is often $t$-averaged with $t$ close to $\frac12$.
By Lemma~\ref{lem:avaraged}, we see that in particular 
for $t = \frac12$
the oracle denoiser in \eqref{eq:oracle_denoise} coincides with our original denoiser
$ 
\mathcal D=I-\gamma \Psi
$
from \eqref{eq:denoiser}, so that it is averaged for this setting.

\section{Numerical Examples of PnP Algorithms with scaled cPNNs} \label{sec:numerics}
Finally, we demonstrate the performance of denoisers built from scaled cPNNs trained as in the previous section within the PnP framework.
As scaling parameter we choose $\gamma = 1.99 < 2$ to ensure convergence of the oracle denoiser within FBS-PnP.
We start with pure denoising, but with varying noise levels.
One advantage of PnP methods is that they achieve good results for a whole range of noise levels, even though the denoiser is only trained for one particular value.  
Then, we present deblurring results. 
As already mentioned, we observe numerically that our learned networks $\Psi$ are $t$-averaged with $t$ close to $0.5$.
The actual estimation of $t$ for $\Psi$ is described in the following remark.

\begin{remark} (Numerical averaging parameter)  \label{rem:estimating_alpha}
An operator $T\colon\R^m\to\R^m$ is $t$-averaged if and only if 
$T=(1-t)I_m+t R$ for some non-expansive operator $R$. 
Thus, it suffices to check if $R\coloneqq\tfrac1t T-\tfrac{1-t}{t} I_m$ is non-expansive.
To verify this numerically, we sample points $x_1,\ldots,x_N$, $N = 10^5$, uniformly from $[0,1]^m$. 
Then, we check whether it holds
\begin{align}
\|JR(x_i)\|_2\leq 1, \quad i=1,\ldots,N,\label{nonexp_check}
\end{align}
where $JR$ denotes the Jacobian of $R$ and $\|\cdot\|_2$ is the spectral norm. 
Note that we can use a matrix-free implementation of the power method to approximate this norm.
If \eqref{nonexp_check} does not hold true, this implies that $R$ is actually expansive.
Otherwise, we consider $R$ as numerically non-expansive.
Now, for finding
$$
t^*=\min\{t\in[\tfrac12,1]:T\text{ is }t\text{-averaged}\}
$$
we start with $t=\tfrac12$ and check if \eqref{nonexp_check} is fulfilled. If this is the case, we set $t^*=t$. Otherwise, we increase $t$ by $0.05$ and repeat this procedure.
For denoising applications, we observed that the estimated $t^*$ is closer to $\tfrac12$ for smaller noise levels.
This is not surprising as a mapping predicting the noise should be more contractive for small noise levels.
\end{remark}

\paragraph{PnP-Denoising}
Here, we apply our cPNNs from the previous section trained for the noise level $\sigma = 25/255 \approx 0.098$ and $\gamma = 1.99$.
Using the procedure in Remark~\ref{rem:estimating_alpha}, we obtain that the numerical averaging parameter of $\Psi$ is $t=0.6$.
We apply the oracle denoiser \eqref{eq:oracle_denoise} with $t = 0.6$, which is averaged by Lemma \ref{lem:avaraged}, and a parameter $\eta\in(0,2)$ optimized via grid search. 
As oracle we use the output of BM3D \cite{DFKE2007} applied to our original noisy image.
For this purpose, the noise level parameter of BM3D is adapted to the corresponding level.
If the noise level is unknown, we can estimate it as described in \cite{SDA2015}.
Note that the data fidelity term $f\colon \R^m\to\R$ in \eqref{eq:sum} is chosen as 
$f(y)=\tfrac12\|x-y\|^2$ such that $\nabla f$ has Lipschitz constant 1 and
by Proposition \ref{conv_ADMM_PnP}i) FBS-PnP converges for $\eta\in(0,2)$.

Now, we denoise images from the BSD68 test set 
corrupted by Gaussian noise with different noise levels $\sigma\in\{0.075,0.1,0.125,0.15\}$. 
As starting iterate within FBS-PnP we use the noisy image.
The resulting PSNR values for the optimal $\eta$ are given in Table~\ref{tab:results_PnP_denoising}.
As comparison, we include the average PSNRs obtained with BM3D and the variational network proposed in \cite{EKKP2020}.
Further, we repeat the experiment with $\gamma=5$ and the denoiser $\mathcal D=I-\gamma \Psi$ without an oracle.
Even though we cannot apply Lemma~\ref{lem:avaraged} to show convergence of FBS-PnP in this case, we observe it numerically.
In Figure~\ref{fig_pnp_noise} we included an example image.
Note that the result obtained with BM3D appears to be more blurred than the FBS-PnP result.
\begin{table}[t]
\centering
\begin{tabular}{c|cccc}
Method&$\sigma=0.075$,&$\sigma=0.1$,&$\sigma=0.125$,&$\sigma=0.15$\\
&$\eta=1.35$&$\eta=0.93$&$\eta=0.72$&$\eta=0.58$\\\hline
Noisy images&$22.50$&$20.00$&$18.06$&$16.48$\\
FBS-PnP with cPNN, $\gamma=1.99$&$30.12$&$28.80$&$27.82$&$27.06$\\
FBS-PnP with cPNN, $\gamma=5\phantom{.00}$&$30.25$&$28.91$&$27.92$&$27.12$\\
Variational network \cite{EKKP2020}&$30.05$&$28.72$&$27.72$&$26.95$\\
BM3D \cite{DFKE2007}&$29.88$&$28.50$&$27.50$&$26.73$
\end{tabular}
\caption{Average PSNR values for denoising images from the BSD68 test set with Gaussian noise for different noise levels $\sigma$ using cPNNs within the FBS-PnP, a variational network and BM3D.}
\label{tab:results_PnP_denoising}
\end{table}

\paragraph{PnP-Deblurring}
Finally, we want to apply cPNNs within PnP for image deblurring. 
To ensure comparability of the results, we use the same setting as in \cite{EKKP2020}. 
That is, we generate the blurred images by applying a blur operator 
$B\colon \R^{d_1,d_2}\to\R^{d_1-8,d_2-8}$ defined as convolution with the normalized version of the kernel $k\in\R^{9,9}$ given by
$$
k_{ij}=\exp\biggl(-\frac{i^2+j^2}{2\tau^2}\biggr),\quad i,j\in\{-4,\ldots,4\}
$$
and adding Gaussian noise with standard deviation $\sigma=0.01$.
As denoiser for PnP we choose $\mathcal D = I-1.99\Psi$ with a cPNN $\Psi$, which is trained for the noise level $\sigma = 0.005$.
For the cPNN $\Psi$, the estimated averaging constant is $t=0.5$.
Consequently, also the resulting denoiser $\mathcal D$ is firmly non-expansive by Lemma~\ref{lem:avaraged}.
According to Proposition~\ref{conv_ADMM_PnP}, FBS-PnP (Algorithm \ref{alg:FBS}) converges for $f\colon \R^d\to\R$ with L-Lipschitz gradient and any $\eta\in(0,\tfrac2L)$.
Further, ADMM-PnP (Algorithm \ref{alg:ADMM}) converges for any $f\in\Gamma_0(\R^m)$.
Note that this enables us to apply ADMM-PnP also for tasks where the data fidelity term is non-smooth, see \cite{BSS2017}.
For our concrete problem, we choose the data fidelity term $f(y)=\tfrac12\|By-x\|^2$, 
where $x$ is the blurred and noisy observation.
Since this function has a $1$-Lipschitz gradient, FBS-PnP converges for any $\eta\in(0,2)$.

Now, we apply both FBS-PnP and ADMM-PnP for reconstructing the original images 
from the blurred ones in the BSD68 test set with blur factor $\tau\in\{1.25,1.5,1.75,2.0\}$. 
Here, we optimize $\eta\in(0,2)$ via grid search and observe 
that the optimal $\eta$ is given by $\eta=1.9$ for FBS-PnP and by $\eta=0.52$ for ADMM-PnP independent of $\tau$.
For comparison, we also include results obtained by the $L_2$-TV model  \cite{ROF1992}, 
where the parameters are also optimized via grid search.
The resulting PSNR values are given in Table~\ref{tab:results_PnP_deblurring} 
and an example image is provided in Figure~\ref{fig_pnp_blur}.
As expected, FBS-PnP and ADMM-PnP yield similar results.
If we have a closer look at the reconstruction obtained with the $L_2$-TV model, we observe the typical stair casing effect.
Such problems do not occur for the PnP based approach.

\begin{figure}[p]
	\begin{subfigure}[t]{0.25\textwidth}
		\includegraphics[width=\textwidth]{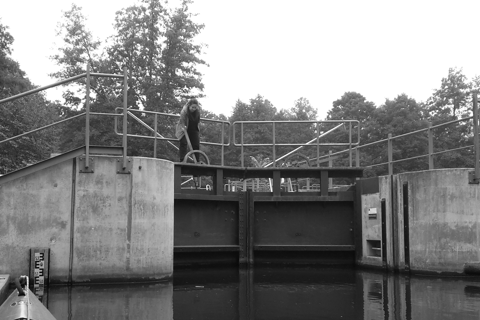}
	\end{subfigure}\hfill
	\begin{subfigure}[t]{0.25\textwidth}
		\includegraphics[width=\textwidth]{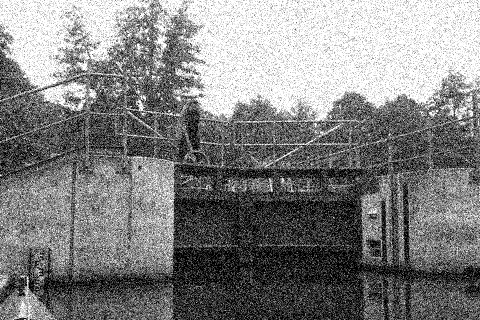}
	\end{subfigure}\hfill
	\begin{subfigure}[t]{0.25\textwidth}
		\includegraphics[width=\textwidth]{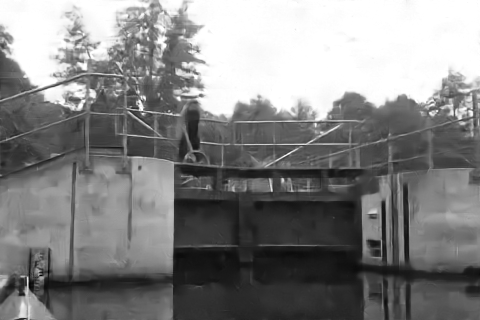}
	\end{subfigure}\hfill
	\begin{subfigure}[t]{0.25\textwidth}
		\includegraphics[width=\textwidth]{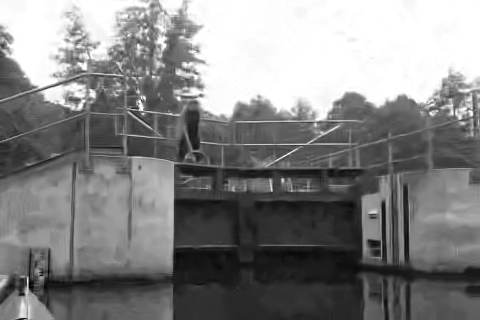}
	\end{subfigure}
	\begin{subfigure}[t]{0.25\textwidth}
		\includegraphics[width=\textwidth]{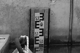}
		\caption*{\scriptsize Original}
	\end{subfigure}\hfill
	\begin{subfigure}[t]{0.25\textwidth}
		\includegraphics[width=\textwidth]{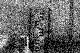}
		\caption*{\scriptsize Noisy\\PSNR $16.44$}
	\end{subfigure}\hfill
	\begin{subfigure}[t]{0.25\textwidth}
		\includegraphics[width=\textwidth]{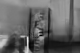}
		\caption*{\scriptsize FBS-PnP with cPNN\\PSNR $27.00$}
	\end{subfigure}\hfill
	\begin{subfigure}[t]{0.25\textwidth}
		\includegraphics[width=\textwidth]{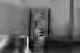}
		\caption*{\scriptsize BM3D\\PSNR $26.62$}
	\end{subfigure}
	\caption{Denoising results with Gaussian noise with standard deviation $\sigma = 0.15$ using FBS-PnP with a cPNN as denoiser and BM3D.}
	\label{fig_pnp_noise}
\end{figure}
\begin{table}[p]
\centering
\begin{tabular}{c|cccc}
Method&$\tau=1.25$,&$\tau=1.5$,&$\tau=1.75$,&$\tau=2.0$\\\hline
Blurred images&$26.46$&$25.60$&$24.98$&$24.53$\\
FBS-PnP with cPNN&$29.78$&$28.62$&$27.70$&$26.98$\\
ADMM-PnP with cPNN&$29.78$&$28.61$&$27.70$&$26.96$\\
Variational network \cite{EKKP2020}&$29.95$&$28.76$&$27.87$&$27.13$\\
$L_2$-TV, $\lambda=0.001$&$29.14$&$28.08$&$27.22$&$26.53$
\end{tabular}
\caption{Average PSNRs for deblurring images from the BSD68 test set with different blur factors $\tau$ using various methods.}
\label{tab:results_PnP_deblurring}
\end{table}
\begin{figure}[p]
\begin{subfigure}[t]{0.25\textwidth}
\includegraphics[width=\textwidth]{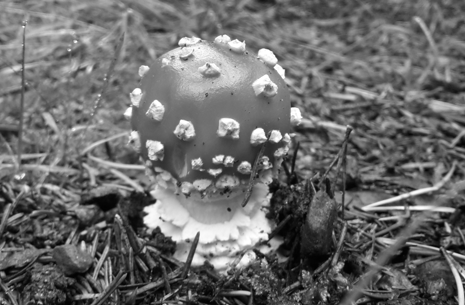}
\end{subfigure}\hfill
\begin{subfigure}[t]{0.25\textwidth}
\includegraphics[width=\textwidth]{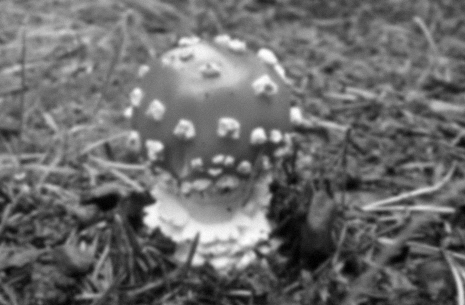}
\end{subfigure}\hfill
\begin{subfigure}[t]{0.25\textwidth}
\includegraphics[width=\textwidth]{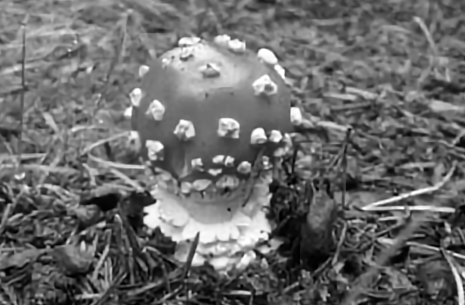}
\end{subfigure}\hfill
\begin{subfigure}[t]{0.25\textwidth}
\includegraphics[width=\textwidth]{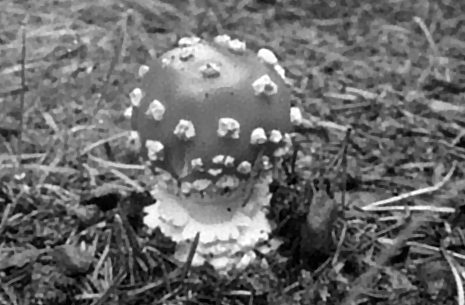}
\end{subfigure}
\begin{subfigure}[t]{0.25\textwidth}
\includegraphics[width=\textwidth]{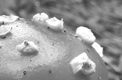}
\caption*{\scriptsize Original}
\end{subfigure}\hfill
\begin{subfigure}[t]{0.25\textwidth}
\includegraphics[width=\textwidth]{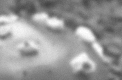}
\caption*{\scriptsize Blurred\\PSNR $25.21$}
\end{subfigure}\hfill
\begin{subfigure}[t]{0.25\textwidth}
\includegraphics[width=\textwidth]{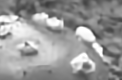}
\caption*{\scriptsize FBS-PnP with cPNN\\PSNR $30.51$}
\end{subfigure}\hfill
\begin{subfigure}[t]{0.25\textwidth}
\includegraphics[width=\textwidth]{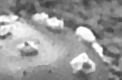}
\caption*{\scriptsize $L_2$-TV\\PSNR $29.77$}
\end{subfigure}
\caption{Deblurring results with blur factor $\tau=1.5$ and Gaussian noise with $\sigma = 0.01$ using FBS-PnP with a cPNN based denoiser and $L_2$-TV.}
\label{fig_pnp_blur}
\end{figure}

\section{Conclusions} \label{sec:conclusions}
In this paper, we extended the PNNs proposed in \cite{HHNPSS2019} to convolutional ones.
For filters with full length, this results in a similar manifold structure as for PNNs and hence a stochastic gradient descent algorithm on some submanifold of the Stiefel manifold can be used to learn these networks.
Unfortunately, this is not true anymore for filters with limited length and a different approach is necessary.
More precisely, we minimized functionals that incorporate an approximation of the ortogonality constraint on $T$ via the Frobenius norm $\|I - T^\tT T\|_F^2$.
At the end of the minimization procedure, a feasible candidate satisfying the constraint is obtained via projection.
We demonstrated how cPNNs can be trained for denoising and observed that the scaling parameter, which is an upper bound for the Lipschitz constant of the network, crucially influences the denoising results.
These observations are exploited within a PnP framework, where we established various convergence guarantees.
Possibly, the provided convergence results can be extended using the theory of almost non-expansive mappings, see \cite{LTT18}.
In the future, we want to exploit the property that we have access to the Lipschitz constant of our scaled cPNNs in other settings, e.g., towards stable invertible NNs, see \cite{BVWG20,HN2020}.

\section*{Acknowledgment}
Funding by the German Research Foundation (DFG) within project STE 571/16-1 and  by the DFG under Germany's Excellence Strategy – The Berlin Mathematics Research Center MATH+ (EXC-2046/1, Projektnummer: 390685689)
is acknowledged.

\appendix

\section{Activation Functions \label{act} }
Table~\ref{prox:act} contains various common activation functions, which are proximal operators.
{\scriptsize
\begin{table}[t]
{\scriptsize
\begin{center}
\begin{tabular}{| l | c | c | }
\hline
 Name & $\sigma_\alpha(x)$ &  $f_\alpha(x)$ 
\\
\hline
\hline
\begin{tabular}{ll}
i) &Linear activation 
\end{tabular}
& $x$ & $0$
		\\[0.5ex]               
		\hline
 \shortstack[c]{\\[0.1ex]\begin{tabular}{ll} ii) &Rectified linear unit \\ &(ReLU)\end{tabular}} 
& 
	$\begin{array}{cl}
		x & \text{if\quad$x>0$} \\
        0 & \text{if\quad$x\leq 0$}\\
  \end{array}$ 
	& $\iota_{[0,\infty)}=
  \begin{array}{ll}
		0 & \text{if\quad$x\in [0,\infty)$} \\
        \infty & \text{if\quad$x\notin [0,\infty)$}\\
  \end{array}$
  \\[0.5ex]
	\hline
 	\shortstack[c]{\\[0.2ex]\begin{tabular}{ll}iii) &Parametric rectified\\ &linear unit (pReLU) \end{tabular}} &  
	$
    \begin{array}{cl}
		x & \text{if\quad$x>0$} \\
        \alpha x & \text{if\quad$x\leq 0$}\\
  \end{array}$ , 
	$\alpha\in [0,1]$ & $
   \begin{array}{cl}
		0 & \text{if\quad$x>0$} \\
        (\frac{1}{\alpha}-1)x^2/2 & \text{if\quad$x\leq 0$}\\
  \end{array}$                                 
	\\[2ex]
	\hline 
		\shortstack[c] {\\[0.8ex]\begin{tabular}{ll} iv) &Saturated linear\\ &activation (SaLU) \end{tabular}
		} &  
		$		
    \begin{array}{cl}
		\alpha & \text{if\quad$x>\alpha$} \\
        x & \text{if\quad$-\alpha\leq x\leq \alpha$}\\
        -\alpha & \text{if\quad$x<-\alpha$}
  \end{array}
	$ 
	& $\iota_{[-1,1]}=
  \begin{cases}
		0 & \text{if\quad$x\in [-1,1]$} \\
        \infty & \text{if\quad$x\notin [-1,1]$}\\
  \end{cases}$
  \\[0.5ex]
	\hline
	\shortstack[c]{\\[0.2ex] \begin{tabular}{ll}v) &Bent identity\\ &activation\end{tabular}} & 
	$ \dfrac{x+\sqrt{x^2+\alpha^2}}{2}$ & $
  \begin{array}{cl}
		x/2-\ln(x+\frac{1}{2})/4 & \text{if\quad$x>-1/2$} \\
        \infty & \text{if\quad$x\leq -1/2$}\\
  \end{array}$
	\\[0.5ex]
	\hline
	  \shortstack[c]{\\[2.3ex]\begin{tabular}{ll}vi) &Soft Thresholding \end{tabular}} &
		$\begin{array}{cl}
		x-\alpha & \text{if\quad$x> \alpha$} \\
		0 & \text{if\quad$x \in [-\alpha,\alpha]$}\\
                x+\alpha & \text{if\quad$x < -\alpha$}  \\
  \end{array}$ 
	& $\alpha |x|$	
		\\[0.5ex]\hline
  	\shortstack[c]{\\[1ex]\begin{tabular}{ll}vii) &Elliot activation \end{tabular}}& 
	$\dfrac{x}{|\alpha x |+1}$ & $
  \begin{array}{cl}
		-|x|-\ln(1-|x|)-x^2/2 & \text{if\quad$|x|<1$} \\
        \infty & \text{if\quad$|x|\geq 1$}\\
  \end{array}$
	    \\[0.5ex]\hline  
		\shortstack[c]{\\[0.2ex] \begin{tabular}{ll}viii) &Inverse square\\ &root unit\end{tabular}} & $
        \dfrac{x}{\sqrt{(\alpha x)^2+1}}$ & $
  \begin{array}{cl}
		-x^2/2-\sqrt{1-x^2} & \text{if\quad$|x|\leq 1$} \\
        \infty & \text{if\quad$|x|>1$}\\
  \end{array}$
	\\[2.2ex]
	\hline	
  \shortstack[c]{\begin{tabular}{ll}ix) &Inverse square root \\& linear unit\end{tabular}} & $
    \begin{array}{cl}
		x & \text{if\quad$x\geq 0$} \\
        \dfrac{x}{\sqrt{(\alpha x)^2+1}} & \text{if\quad$x<0$}\\
  \end{array}$ & $
  \begin{array}{cl}
  		0 & \text{if\quad$x\geq 0$} \\
		1-x^2/2-\sqrt{1-x^2} & \text{if\quad$-1\leq x<0$} \\
        \infty & \text{if\quad$x<-1$}\\
  \end{array}$
		\\
	\hline
	\end{tabular}
	\caption{Stable activation functions $\sigma_\alpha$ and $f_\alpha(x)$ with $\sigma_\alpha = \prox_{f_\alpha}$,	where $\alpha$ can be skipped if the function is parameter free, see \cite{CP2018}.\label{prox:act}}
\end{center}
}
\end{table}
}

\section{Proof of Proposition \ref{conv_ADMM_PnP} }
\begin{proof}
	i)
	As the concatenation of averaged operators is averaged, see Theorem~\ref{alpha_lin}, and the iteration sequence generated by an averaged operator
	converges globally, it suffices to show that $I-\eta \nabla f$ is averaged. 
	By the Baillon--Haddad theorem \cite[Cor.~16.1]{BC2011},
	we get that $\tfrac1L\nabla f$ is firmly non-expansive, 
	i.e., there exists some non-expansive mapping $R$ such that $\tfrac1L\nabla f=\tfrac12(I+R)$. 
	Thus, it holds
	$$
	I-\eta\nabla f =I-\tfrac{\eta L}{2}(I+R) = (1-\tfrac{\eta L}{2})I+\tfrac{\eta L}{2}(-R).
	$$
	Consequently, the operator $I-\eta\nabla f$ is averaged for $\eta \in(0,\tfrac2L)$. 
	\\
	ii)
	Defining
	\begin{equation}\label{iter}
		t^{(r+1)} \coloneqq  \tfrac{1}{\gamma} p^{(r)} +  x^{(r+1)},
	\end{equation}
	we can rewrite the second step of the ADMM-PnP algorithm as
	\begin{equation} \label{yps}
		y^{(r+1)} = \Psi \big( t^{(r+1)} \big).
	\end{equation}
	Then, the third step can be written as
	\begin{align}
		p^{(r+1)}&= p^{(r)} + \gamma\bigl( x^{(r+1)} -  y^{(r+1)}\bigr) = p^{(r)} + \gamma x^{(r+1)} - \gamma \Psi \bigl( t^{(r+1)}\bigr) \nonumber\\
		&= \gamma \bigl( t^{(r+1)} -  \Psi( t^{(r+1)} ) \bigr). \label{pe}
	\end{align}
	Further, the first step
	\begin{equation} \label{fundament_1}
		x^{(r+1)} = \prox_{\tfrac1\gamma f}\bigl(y^{(r)} -  \tfrac{1}{\gamma}p^{(r)}\bigr)
	\end{equation}
	can be rewritten using the Moreau decomposition and \eqref{iter} as
	\begin{align}
		x^{(r+1)} &=  y^{(r)} -  \tfrac{1}{\gamma}p^{(r)} - \tfrac{1}{\gamma} \prox_{\gamma f^*}\bigl(\gamma y^{(r)} - p^{(r)}\bigr)\notag\\
		& = 2\Psi\bigl(t^{(r)}\bigr) -  t^{(r)} - \tfrac{1}{\gamma} \prox_{\gamma f^*}\bigl(2 \gamma \Psi(t^{(r)}) - \gamma t^{(r)}\bigr).\label{x}
	\end{align}
	Plugging \eqref{pe} and \eqref{x} into \eqref{iter}, we conclude
	\begin{align*}
		t^{(r+1)} & = t^{(r)} - \Psi\bigl(t^{(r)}\bigr) + 2\Psi\bigl(t^{(r)}\bigr) -  t^{(r)} - \tfrac{1}{\gamma} \prox_{\gamma f^*} \bigl(2 \gamma \Psi(t^{(r)}) - \gamma t^{(r)}\bigr)\\
		&=  \Psi\bigl(t^{(r)}\bigr) - \tfrac{1}{\gamma} \prox_{\gamma f^*} \bigl( 2 \gamma \Psi(t^{(r)}) -  \gamma t^{(r)} \bigr).
	\end{align*}
	With 
	$$R_1 \coloneqq \gamma \left(2 \Psi -  I \right), \quad R_2 \coloneqq I - 2 \prox_{\gamma f^*}$$
	this results in
	\begin{equation}
		t^{(r+1)} = \underbrace{\tfrac12 \bigl( I + \tfrac{1}{\gamma} R_2 \circ R_1 \bigr)}_{\mathcal T} ( t^{(r)} ).
	\end{equation}
	Since $\Psi$ and the proximity operator are $\frac12$-averaged,
	we know that the reflections $R_1/\gamma$ and $R_2$ are non-expansive. 
	Thus, ${\mathcal T}$ is also firmly non-expansive and the sequence $\{t^{(r)}\}_r$ converges globally.
	As both $\Psi$ and $\prox_{\gamma f^*}$ are continuous, the same holds true by \eqref{yps}, \eqref{pe} and \eqref{x} for $\{y^{(r)}\}_r$, $\{p^{(r)}\}_r$
	and $\{x^{(r)}\}_r$. 
\end{proof}
\bibliographystyle{abbrv}
\bibliography{references}

\begin{thebibliography}{10}

\bibitem{AMS08}
P.-A. Absil, R.~Mahony, and R.~Sepulchre.
\newblock {\em Optimization Algorithms on Matrix Manifolds}.
\newblock Princeton University Press, 2008.

\bibitem{BC2011}
H.~H. Bauschke and P.~L. Combettes.
\newblock {\em Convex Analysis and Monotone Operator Theory in Hilbert Spaces}.
\newblock Springer, New York, 2011.

\bibitem{Beck17}
A.~Beck.
\newblock {\em First-Order Methods in Optimization}, volume~25 of {\em MOS-SIAM
  Series on Optimization}.
\newblock SIAM, Philadelphia, 2017.

\bibitem{BVWG20}
J.~Behrmann, P.~Vicol, K.-C. Wang, R.~Grosse, and J.-H. Jacobsen.
\newblock Understanding and mitigating exploding inverses in invertible neural
  networks.
\newblock {\em ArXiv preprint arXiv:2006.09347}, 2020.

\bibitem{bini1983spectral}
D.~Bini and M.~Capovani.
\newblock Spectral and computational properties of band symmetric {T}oeplitz
  matrices.
\newblock {\em Linear Algebra Appl.}, 52/53:99--126, 1983.

\bibitem{BPCPE11}
S.~Boyd, N.~Parikh, E.~Chu, B.~Peleato, and J.~Eckstein.
\newblock Distributed optimization and statistical learning via the alternating
  direction method of multipliers.
\newblock {\em Found. Trends Mach. Learn.}, 3(1):101--122, 2011.

\bibitem{Braides02}
A.~Braides.
\newblock {\em {$\Gamma$}-Convergence for Beginners}.
\newblock Oxford University Press, Oxford, 2002.

\bibitem{BSS2017}
M.~Burger, A.~Sawatzky, and G.~Steidl.
\newblock First order algorithms in variational image processing.
\newblock In {\em Operator Splittings and Alternating Direction Methods}.
  Springer, 2017.

\bibitem{BX08}
R.~Byers and H.~Xu.
\newblock A new scaling for {N}ewton's iteration for the polar decomposition
  and its backward stability.
\newblock {\em SIAM J. Matrix Anal. Appl.}, 30(2):822--843, 2008.

\bibitem{CWE2016}
S.~H. Chan, X.~Wang, and O.~A. Elgendy.
\newblock Plug-and-play {ADMM} for image restoration: Fixed-point convergence
  and applications.
\newblock {\em IEEE Trans. Comput. Imaging}, 3:84--98, 2016.

\bibitem{CP2018}
P.~L. Combettes and J.-C. Pesquet.
\newblock Deep neural network structures solving variational inequalities.
\newblock {\em Set-Valued Var. Anal.}, 28:491--518, 2020.

\bibitem{CW05}
P.~L. Combettes and V.~R. Wajs.
\newblock Signal recovery by proximal forward-backward splitting.
\newblock {\em Multiscale Model. Simul.}, 4(4):1168--1200, 2005.

\bibitem{CY15}
P.~L. Combettes and I.~Yamada.
\newblock Compositions and convex combinations of averaged nonexpansive
  operators.
\newblock {\em J. Math. Anal. Appl.}, 425(1):55--70, 2015.

\bibitem{CKCH2020}
L.~Condat, D.~Kitahara, A.~Contreras, and A.~Hirabayashi.
\newblock Proximal splitting algorithms: Relax them all!
\newblock {\em ArXiv Preprint arXiv:1912.00137}, 2019.

\bibitem{DFKE2007}
K.~Dabov, A.~Foi, V.~Katkovnik, and K.~Egiazarian.
\newblock Image denoising by sparse {3D} transform-domain collaborative
  filtering.
\newblock {\em IEEE Trans. Image Process.}, 16(8):2080--2095, 2007.

\bibitem{DKE2012}
A.~Danielyan, V.~Katkovnik, and K.~Egiazarian.
\newblock {BM3D} frames and variational image deblurring.
\newblock {\em IEEE Trans. Image Process.}, 21(4):1715--1728, 2012.

\bibitem{DDM04}
I.~Daubechies, M.~Defrise, and C.~De~Mol.
\newblock An iterative thresholding algorithm for linear inverse problems with
  a sparsity constraint.
\newblock {\em Commun. Pure Appl. Math.}, 57(11):1413--1457, 2004.

\bibitem{DMM2009}
D.~L. Donoho, A.~Maleki, and A.~Montanari.
\newblock Message-passing algorithms for compressed sensing.
\newblock {\em Proc. Natl. Acad. Sci.}, 106(45):18914--18919, 2009.

\bibitem{EB92}
J.~Eckstein and D.~P. Bertsekas.
\newblock On the {D}ouglas-{R}achford splitting method and the proximal point
  algorithm for maximal monotone operators.
\newblock {\em Math. Program.}, 55:293--318, 1992.

\bibitem{EKKP2020}
A.~Effland, E.~Kobler, K.~Kunisch, and T.~Pock.
\newblock Variational networks: an optimal control approach to early stopping
  variational methods for image restoration.
\newblock {\em J. Math. Imaging Vis.}, 62(3):396--416, 2020.

\bibitem{GFPC18}
H.~Gouk, E.~Frank, B.~Pfahringer, and M.~Cree.
\newblock Regularisation of neural networks by enforcing {L}ipschitz
  continuity.
\newblock {\em ArXiv preprint arXiv:1804.04368}, 2018.

\bibitem{GJNMU2018}
H.~Gupta, K.~H. Jin, H.~Q. Nguyen, M.~T. McCann, and M.~Unser.
\newblock {CNN}-based projected gradient descent for consistent {CT} image
  reconstruction.
\newblock {\em IEEE Trans. Med. Imaging}, 37(6):1440--1453, 2018.

\bibitem{HN2020}
P.~Hagemann and S.~Neumayer.
\newblock Stabilizing invertible neural networks using mixture models.
\newblock {\em ArXiv preprint :2009.02994}, 2020.

\bibitem{HHNPSS2019}
M.~Hasannasab, J.~Hertrich, S.~Neumayer, G.~Plonka, S.~Setzer, and G.~Steidl.
\newblock Parseval proximal neural networks.
\newblock {\em J. Fourier Anal. Appl.}, 26:59, 2020.

\bibitem{HS2020}
J.~Hertrich and G.~Steidl.
\newblock Inertial stochastic {PALM} and its application for learning
  {S}tudent-$t$ mixture models.
\newblock {\em ArXiv preprint arXiv:2005.02204}, 2020.

\bibitem{High86}
N.~J. Higham.
\newblock Computing the polar decomposition--with applications.
\newblock {\em SIAM J. Sci. Statist. Comput.}, 7(4):1160--1174, 1986.

\bibitem{Higham08}
N.~J. Higham.
\newblock {\em Functions of Matrices: Theory and Computation}.
\newblock SIAM, Philadelphia, 2008.

\bibitem{HS2013}
R.~A. Horn and C.~R. Johnson.
\newblock {\em Matrix Analysis}.
\newblock Oxford University Press, 2013.

\bibitem{huang2018orthogonal}
L.~Huang, X.~Liu, B.~Lang, A.~W. Yu, Y.~Wang, and B.~Li.
\newblock Orthogonal weight normalization: Solution to optimization over
  multiple dependent {S}tiefel manifolds in deep neural networks.
\newblock In {\em 32nd AAAI Conference on Artificial Intelligence}, 2018.

\bibitem{KB2014}
D.~P. Kingma and J.~Ba.
\newblock Adam: A method for stochastic optimization.
\newblock {\em ArXiv preprint arXiv:1412.6980}, 2014.

\bibitem{Kr55}
M.~A. Krasnoselskii.
\newblock Two observations about the method of successive approximations.
\newblock {\em Uspekhi Matematicheskikh Nauk}, 10:123--127, 1955.
\newblock In Russian.

\bibitem{LLT2020}
J.~Li, F.~Li, and S.~Todorovic.
\newblock Efficient {R}iemannian optimization on the {S}tiefel manifold via the
  {C}ayley transform.
\newblock In {\em 8th International Conference on Learning Representations,
  {ICLR} 2020, Addis Abeba, Ethiopia, April 26-30, 2020}, 2020.

\bibitem{Ma53}
W.~R. Mann.
\newblock Mean value methods in iteration.
\newblock {\em Proc. Amer. Math. Soc.}, 16(4):506--510, 1953.

\bibitem{MFTM2001}
D.~Martin, C.~Fowlkes, D.~Tal, and J.~Malik.
\newblock A database of human segmented natural images and its application to
  evaluating segmentation algorithms and measuring ecological statistics.
\newblock In {\em Proc. Eighth IEEE International Conference on Computer
  Vision. ICCV 2001}, volume~2, pages 416--423. IEEE, 2001.

\bibitem{MMHC2017}
T.~Meinhardt, M.~Moeller, C.~Hazirbas, and D.~Cremers.
\newblock Learning proximal operators: {U}sing denoising networks for
  regularizing inverse imaging problems.
\newblock In {\em Proc. IEEE International Conference on Computer Vision},
  pages 1799--1808, 2017.

\bibitem{MKKY2018}
T.~Miyato, T.~Kataoka, M.~Koyama, and Y.~Yoshida.
\newblock Spectral normalization for generative adversarial networks.
\newblock In {\em International Conference on Learning Representations}, 2018.

\bibitem{MLE19}
V.~Monga, Y.~Li, and Y.~Eldar.
\newblock Algorithm unrolling: Interpretable, efficient deep learning for
  signal and image processing.
\newblock {\em ArXiv Preprint arXiv:1912.10557}, 2019.

\bibitem{Moreau65}
J.-J. Moreau.
\newblock Proximit\'{e} et dualit\'{e} dans un espace {H}ilbertien.
\newblock {\em Bulletin de la Soci\'{e}t\'{e} Math\'{e}matique de France},
  93:273--299, 1965.

\bibitem{NA2005}
Y.~Nishimori and S.~Akaho.
\newblock Learning algorithms utilizing quasi-geodesic flows on the {S}tiefel
  manifold.
\newblock {\em Neurocomputing}, 67:106--135, 2005.

\bibitem{Ono2017}
S.~Ono.
\newblock Primal-dual plug-and-play image restoration.
\newblock {\em IEEE Signal Process. Lett.}, 24(8):1108--1112, 2017.

\bibitem{PS99}
D.~Potts and G.~Steidl.
\newblock Preconditioners for ill--conditioned {T}oeplitz matrices.
\newblock {\em BIT}, 39(3):513--533, 1999.

\bibitem{CLPKS2017}
J.~Rick~Chang, C.-L. Li, B.~Poczos, B.~Vijaya~Kumar, and A.~C.
  Sankaranarayanan.
\newblock One network to solve them all--solving linear inverse problems using
  deep projection models.
\newblock In {\em Proc. of the IEEE International Conference on Computer
  Vision}, pages 5888--5897, 2017.

\bibitem{REM2017}
Y.~Romano, M.~Elad, and P.~Milanfar.
\newblock The little engine that could: {R}egularization by denoising ({RED}).
\newblock {\em SIAM J. Imaging Sci.}, 10(4):1804--1844, 2017.

\bibitem{ROF1992}
L.~Rudin, S.~Osher, and E.~Fatemi.
\newblock Nonlinear total variation based noise removal algorithms.
\newblock {\em Physica D}, 60:259--268, 1992.

\bibitem{LTT18}
D.~Russell~Luke, N.~H. Thao, and M.~K. Tam.
\newblock Quantitative convergence analysis of iterated expansive, set-valued
  mappings.
\newblock {\em Math. Oper. Res.}, 43(4):1143--1176, 2018.

\bibitem{SGL2019}
H.~Sedghi, V.~Gupta, and P.~M. Long.
\newblock The singular values of convolutional layers.
\newblock In {\em International Conference on Learning Representations}, 2019.

\bibitem{Se2011}
S.~Setzer.
\newblock Operator splittings, {B}regman methods and frame shrinkage in image
  processing.
\newblock {\em Int. J. Comput. Vis.}, 92(3):265--280, 2011.

\bibitem{SKM2019}
H.~Sommerhoff, A.~Kolb, and M.~Moeller.
\newblock Energy dissipation with plug-and-play priors.
\newblock In {\em NeurIPS 2019 Workshop}, 2019.

\bibitem{SVW2016}
S.~Sreehariand, S.~V. Venkatakrishnan, and B.~Wohlberg.
\newblock Plug-and-play priors for bright field electron tomography and sparse
  interpolation.
\newblock {\em IEEE Trans. Comput. Imaging}, 2:408--423, 2016.

\bibitem{strang2014functions}
G.~Strang and S.~MacNamara.
\newblock Functions of difference matrices are {T}oeplitz plus {H}ankel.
\newblock {\em SIAM Review}, 56(3):525--546, 2014.

\bibitem{SWK2019}
Y.~Sun, B.~Wohlberg, and U.~S. Kamilov.
\newblock An online plug-and-play algorithm for regularized image
  reconstruction.
\newblock {\em IEEE Trans. Comput. Imaging}, 5(3):395--408, 2019.

\bibitem{SDA2015}
C.~Sutour, C.-A. Deledalle, and J.-F. Aujol.
\newblock Estimation of the noise level function based on a nonparametric
  detection of homogeneous image regions.
\newblock {\em SIAM J. Imaging Sci.}, 8:2622--2661, 2015.

\bibitem{TBF2017}
A.~Teodoro, J.~M. Bioucas-Dias, and M.~Figueiredo.
\newblock Scene-adapted plug-and-play algorithm with convergence guarantees.
\newblock In {\em Proc. IEEE Int. Workshop on Machine Learning for Signal
  Processing}, 2007.

\bibitem{TRPW20}
M.~{Terris}, A.~{Repetti}, J.~{Pesquet}, and Y.~{Wiaux}.
\newblock Building firmly nonexpansive convolutional neural networks.
\newblock In {\em ICASSP 2020 - IEEE International Conference on Acoustics,
  Speech and Signal Processing}, pages 8658--8662, 2020.

\bibitem{TSS2018}
Y.~Tsuzuku, I.~Sato, and M.~Sugiyama.
\newblock Lipschitz-margin training: Scalable certification of perturbation
  invariance for deep neural networks.
\newblock In {\em Advances in Neural Information Processing Systems 31}, pages
  6541--6550. Curran Associates, Inc., 2018.

\bibitem{VBW13}
S.~V. Venkatakrishnan, C.~A. Bouman, and B.~Wohlberg.
\newblock Plug-and-play priors for model based reconstruction.
\newblock In {\em 2013 IEEE Global Conference on Signal and Information
  Processing}, pages 945--948. IEEE, 2013.

\bibitem{WY2013}
Z.~Wen and W.~Yin.
\newblock A feasible method for optimization with orthogonality constraints.
\newblock {\em Math. Program.}, 142(1--2):397--434, 2013.

\bibitem{ZZCMZ2017}
K.~{Zhang}, W.~{Zuo}, Y.~{Chen}, D.~{Meng}, and L.~{Zhang}.
\newblock Beyond a {G}aussian denoiser: Residual learning of deep {CNN} for
  image denoising.
\newblock {\em IEEE Trans. Image Process.}, 26(7):3142--3155, 2017.

\bibitem{ZZGZ2017}
K.~Zhang, W.~Zuo, S.~Gu, and L.~Zhang.
\newblock Learning deep {CNN} denoiser prior for image restoration.
\newblock In {\em Proc. of the IEEE Conference on Computer Vision and Pattern
  Recognition}, pages 3929--3938, 2017.

\end{thebibliography}
\end{document}